\documentclass{elsarticle}
\usepackage{a4wide, amstext, amsmath, amssymb, graphicx, array, epsfig, psfrag}
\usepackage{amstext, amsmath, amssymb, graphicx, todonotes}
\usepackage{stmaryrd}
\usepackage{color}
\usepackage[update,prepend]{epstopdf}

\newcommand{\bn}{\boldsymbol n}

\newcommand{\bx}{\boldsymbol x}

\newcommand{\by}{\boldsymbol y}
\newcommand{\bv}{\boldsymbol v}
\newcommand{\bnu}{\boldsymbol\nu}
\newcommand{\bu}{\boldsymbol u}
\newcommand{\bw}{\boldsymbol w}
\newcommand{\be}{\boldsymbol e}

\newcommand{\bV}{\boldsymbol V}
\newcommand{\bW}{\boldsymbol W}
\newcommand{\bSigma}{\boldsymbol \Sigma}
\newcommand{\ff}{\boldsymbol f}
\newcommand{\blambda}{\boldsymbol \lambda}
\newcommand{\bvarphi}{\boldsymbol \varphi}

\newcommand{\bmu}{\boldsymbol \mu}

\newcommand{\bfeta}{\boldsymbol \eta}

\newcommand{\bfzeta}{\boldsymbol \zeta}
\newcommand{\bfxi}{\boldsymbol \xi}
\newcommand{\srf}{\bfzeta}
\newcommand{\ssrf}{\bfxi}

\newcommand{\bsigma}{\boldsymbol{\sigma}}

\newcommand{\mcS}{\mathcal{S}}

\newcommand{\IR}{\mathbb{R}}

\newcommand{\tnorm}[1]{\left\Vert{\hskip -2.6pt}\left\vert #1 \right\vert{\hskip -2.6pt}\right\Vert}

\newcommand{\tT}{\tilde{T}}
\newcommand{\tF}{\tilde{F}}
\newcommand{\tbx}{\tilde{\bx}}
\newcommand{\tby}{\tilde{\by}}
\newcommand{\tbvarphi}{\tilde{\bvarphi}}
\newcommand{\tvarphi}{\tilde{\varphi}}
\numberwithin{equation}{section}

\newtheorem{prop}{Proposition}[section]
\newtheorem{thm}{Theorem}[section]
\newtheorem{rem}{Remark}[section]
\newtheorem{assumption}{Assumption}[section]
\newtheorem{cor}{Corollary}[section]
\newenvironment{proof}{\noindent \newline {\bf Proof.}}
{\hfill \mbox{\fbox{} } \newline}

\begin{document}
\title{Cut finite element method for divergence free approximation of incompressible flow: a Lagrange multiplier approach}
\author[UK]{Erik Burman}
\address[UK]{Department of Mathematics, University College London, London, WC1E  6BT,  United Kingdom} 
\ead{e.burman@ucl.ac.uk} 
\author[Jon]{Peter Hansbo}
\address[Jon]{ Department of Materials and Manufacturing, J\"onk\"oping University, SE-55111 J\"onk\"oping, Sweden }
\ead{peter.hansbo@ju.se} 
\author[Ume]{Mats G.\ Larson}
\address[Ume]{Department of Mathematics and Mathematical Statistics, Ume{\aa} University, SE-90187 Ume{\aa}, Sweden} 
\ead{mats.larson@umu.se} 
\begin{abstract}
In this note we design a cut finite element method for a low order divergence free element applied to a boundary value problem subject to Stokes' equations. For the imposition of Dirichlet boundary conditions we consider either Nitsche's method or a stabilized Lagrange multiplier method. In both cases the normal component of the velocity is constrained using a multiplier, different from the standard pressure approximation. The divergence of the approximate velocities is pointwise zero over the whole mesh domain, and we derive optimal error estimates for the velocity and pressures, where the error constant is independent of how the physical domain intersects the computational mesh, and of the regularity of the pressure multiplier imposing the divergence free condition.
\end{abstract}
\maketitle
\section{Introduction}

We consider the Stokes' equations of creeping incompressible flow with homogeneous Dirichlet boundary conditions. 
In this work the aim is to develop a robust and accurate cut finite element method \cite{UCL17, BCHLM15} for the approximation of the Stokes' equations using pointwise divergence free velocities. This means that the computational mesh does not respect the physical geometry, but can cut it in a quite general fashion. The combination of divergence free approximation spaces and CutFEM is known to be problematic due to the coupling of velocity and pressure on the boundary and the perturbations of the incompressibility induced by so called ghost penalty terms that are frequently used to enhance the stability of cut discretizations \cite{Bu10,MLLR14,BHL22}. Unfitted finite element methods for the interface problem in incompressible elasticity was introduced in \cite{BBH09} using Nitsche's method and further developed in \cite{HLZ14, CFGSZ15}. Fictitious domain methods for the Stokes' problem using cut elements were introduced in \cite{BH14,MLLR14}. For an analysis of inf sup stability of unfitted FEM we refer to \cite{GO18}. Unfitted FEM using stabilized Lagrange multiplier methods for Stokes' was discussed in \cite{FL17}, drawing on earlier results from  \cite{HR09} and \cite{BH10}. None of these references treat the case of pointwise divergence free approximation. Only in the recent paper \cite{LNO21} an unfitted finite element method was proposed in this context using Nitsche's method and ghost penalty. Optimal error estimates where shown and also that the solution was pointwise divergence free in the interior, but not up to the boundary. The pressure regularity also polluted the error constant unlike what one expects when using divergence free approximation. For a Darcy flow interface problem, in \cite{FHNZ22}, a new variant of ghost penalty was introduced to allow for cut elements and divergence free approximation. In this case the divergence free condition was satisfied globally.

In this contribution, we will entirely avoid ghost penalty for the velocity and pressure approximations. The idea is to instead impose the divergence-free condition over the whole computational mesh. Hence there is no need for a penalty on the pressure. The boundary condition is applied either using Nitsche's method or a stabilized Lagrange multiplier technique of interior penalty type. In either case, the pressure term typically appearing on the boundary is decoupled from the multiplier imposing the divergence-free constraint. Either as a normal stress variable or as a separate boundary pressure. This makes the imposition of the divergence-free constraint the sole purpose of the pressure variable, unlike the classical Nitsche method for Stokes' equations, where it also appears as a multiplier for the normal velocity component. To ensure pointwise divergence-free approximation this coupling must be broken \cite{BP21, FHNZ22}, resulting in a lack of skew symmetry of the pressure velocity coupling for Nitsche's method. In our Nitsche variant, the bulk and boundary pressures are represented by different variables and hence completely decoupled. The Lagrange multiplier method that we propose is similar to the standard CutFEM using stabilized Lagrange multipliers \cite{BH10, FL17}. For both cases, the key modification is that the velocity pressure coupling terms are integrated over the whole mesh domain; that is, these integrals are not restricted to the physical domain and do not contain any cut elements. This results in a conformity error for the pressure approximation that can be shown to only affect the pressure locally in the case of divergence-free approximation.
Hence a globally accurate pressure approximation can be constructed using post-processing with nearest neighbor extension. To keep down the technicalities, we will work with the minimal element introduced in \cite{CH18}, but with some added effort, the ideas carry over to other divergence-free spaces.
\section{Model problem: the Stokes' equations}\label{sec:models}
Let $\Omega$ be an open subset of $\IR^d$ with smooth, \textcolor{black}{non self-intersecting}, boundary $\Gamma$. Let $\bn_\Gamma$ denote the outward pointing normal on $\Gamma$. We look for a
velocity-pressure couple $(\bu,p) \in \bV_0 \times
Q$, where $\bV_0:= [H^1_0(\Omega)]^d$ and $Q:=L^2_0(\Omega)$ denotes the set of square
integrable functions with mean zero, such that
\begin{equation}\label{eq:brinkman}
\begin{array}{rcl}
-\Delta \bu + \nabla p  &=& \ff \mbox{ in } \Omega, \\
\nabla \cdot \bu & = & 0  \mbox{ in } \Omega,\\
 \bu &=& 0\mbox{ on } \Gamma.
\end{array}
\end{equation}
Here $\ff \in [L^2(\Omega)]^d$. The weak formulation can be written, find $(\bu,p) \in \bV_0 \times Q$ such that
\begin{equation}\label{eq:weakStokes}
a(\bu,\bv) - b(p,\bv) + b(q,\bu) = l(\bv),\mbox{ for all } (\bv,q) \in \bV_0\times
Q,
\end{equation}
where
\begin{equation}\label{eq:rhsB}
l(\bv) := \int_\Omega \ff \cdot \bv ~\mbox{d}x,
\end{equation}
\[
a(\bw,\bv) := \int_\Omega  \nabla \bw : \nabla \bv~\mbox{d}x, \mbox{ with } \nabla \bw : \nabla \bv := \sum_{i=1}^d \nabla w_i \cdot \nabla v_i
\]
and
\[
b(q,\bv):= \int_\Omega q \nabla \cdot \bv ~\mbox{d}x.
\]
The formulation \eqref{eq:weakStokes} admits a unique solution with additional regularity
\begin{equation}
\|\bu\|_{[H^2(\Omega)]^d} + \|p\|_{H^1(\Omega)} \lesssim \|\ff\|_\Omega.
\end{equation}
Here and below we use the notation $a \lesssim b$ for $a \leq C b$, where $C$ is a constant independent of the local mesh size and the mesh-domain configuration.

We will consider two formulations for the discretization of (2.1). In both cases, the boundary condition will be imposed weakly. To separate the effect of the pressure in the bulk from its interaction with the velocity on the boundary, we introduce a Lagrange multiplier that represents the normal stress on the boundary or the pressure part thereof. This can be done in different ways on the discrete level; but,  to give details on the continuous level, it is convenient to introduce a multiplier for all components of the stress. Therefore we will instead look for $\bu \in \bV:=[H^1(\Omega)]^d$ and the multiplier expressing the fluid stress on the solid wall $\blambda \in \bSigma := [H^{-\frac12}(\Gamma)]^d$. Formally, if $\nabla \bu$ is the matrix with columns $\nabla u_i$, $i=1,\hdots,d$,
\[
\blambda =  -( \nabla \bu)^T \bn_\Gamma  + p \bn_\Gamma \mbox{ on } \Gamma.
\]

The corresponding weak formulation reads: find $(\bu,p,\blambda) \in \bV\times
Q \times\bSigma$ such that:
\[
A[(\bu,p,\blambda),(\bv,q,\bmu)] = l(\bv),\mbox{ for all } (\bv,q,\bmu) \in \bV\times
Q\times \bSigma.
\]
Here the bilinear form $A$ is given by
\begin{equation}\label{eq:bilin_Brinkman}
A[(\bu,p,\blambda),(\bv,q,\bmu)] :=a(\bu,\bv) - b(p,\bv) + b(q,\bu) +c(\blambda,\bv) - c(\bu,\bmu)
\end{equation}
with
\[
 c(\bmu, \bv) = \int_{\Gamma} \bv\cdot  \bmu ~\mbox{d}s.
\]
Unique existence of $\bu$ is ensured through the application of the
Lax-Milgram lemma in the space  $
\bV_0 \cap H^{div}_0$, where
\[
H^{div}_0:= \{\bv \in \bV: \nabla \cdot \bv = 0\}.
\]
A unique bulk pressure $p$ and boundary force $\blambda $ is then guaranteed by the
Ladyzhenskaya-Babuska-Brezzi condition \cite{BF91}.
\section{The finite element space}\label{sec:FEMconst}
To define the unfitted finite element method we let $S$ denote an open polytopal domain such that $\bar \Omega \subset S$ and $\mbox{dist}(\partial S, \Gamma) = 1$.
Let $\mathcal{T}_S$ denote a quasi-uniform, conforming, shape regular tesselation 
of simplexes $T$ of $S$. We let $h_T:= \mbox{diam}(T)$. Then we extract the set of elements intersecting the physical domain 
\[
\mathcal{T}_h := \{T \in \mathcal{T}_S \mbox{ such that } \Omega \cap T \ne \emptyset\}.
\]
We define the index $h$ to be the mesh parameter, $h := \max_{T \in \mathcal{T}_h} h_T$.

To distinguish the elements intersected by $\Gamma$ from those in the bulk of $\Omega$ we define
\[
\mathcal{T}_\Gamma := \{T \in \mathcal{T}_h : T \cap \Omega \ne \emptyset\} \mbox{ and } \mathcal{T}_I := \mathcal{T}_h \setminus  \mathcal{T}_\Gamma.
\]
The domains covered by the simplexes in the different sets are then denoted
\[
\Omega_\mathcal{T} := \cup_{T \in \mathcal{T}_h} T,\quad \Omega_\Gamma := \cup_{T \in \mathcal{T}_\Gamma} T,\quad \Omega_I := \cup_{T \in \mathcal{T}_I} T.
\]
The boundary of the mesh domain is denoted $\partial \Omega_\mathcal{T}$.
Clearly $\Omega\subset \Omega_\mathcal{T}$.
\textcolor{black}{Since $\Gamma$ is smooth there exists a tubular neighbourhood of $\Gamma$, $\Gamma \subset U_\delta$ with thickness $\delta$, \textcolor{black}{that does not self intersect}. We assume that $h<\delta$, so that $\Omega_{\Gamma} \subset U_\delta$.}
We denote the set of faces of the
simplexes in the set $\mathcal{T}_X$ by $\mathcal{F}(\mathcal{T}_X)$ and the subset of interior faces $\mathcal{F}_i(\mathcal{T}_X)$, that is faces such that $F = T \cap T'$ for $T,\, T' \in \mathcal{T}_X$.
We let $Q_h$ denote the space of functions in $L^2(\Omega)$ that are
constant on each element,
\[
Q_h := \{x \in L^2(\Omega_\mathcal{T}) : x\vert_T \in \mathbb{P}_0(T); \forall T
\in \mathcal{T}_h \}.
\]
We let $\bW_h$ denote the space of
vectorial piecewise affine functions on $\mathcal{T}_h$,
\[
\bW_h := \{v \in [H^1(\Omega_\mathcal{T})]^d: v\vert_T \in  [\mathbb{P}_1(T)]^d; \forall T
\in \mathcal{T}_h \}.
\] 
To define a finite element space for the Lagrange multiplier we 
define the extended trace space as the restriction of $Q_h$ to the elements in $\mathcal{T}_\Gamma$,
\[
\Sigma_h := \{x \in L^2(\Omega_\Gamma) : x\vert_T \in \mathbb{P}_0(T); \forall T
\in \mathcal{T}_\Gamma \} \mbox{ and } \bSigma_h = [\Sigma_h]^d.
\]
It is well known that the space $\bW_h$
is not robust for nearly incompressible elasticity and that the
velocity-pressure space $\bW_h \times Q_h$ is unstable for
incompressible flow problems on general meshes. 
To rectify this we will enrich the space with vectorial piecewise affine bubbles, defined on a subgrid, following the design in \cite{CH18, BCH20}, that allows us to remain
conforming in $H^1$. The bubbles then allow us to define degrees of freedom on the faces. This results in an extended space, that we will denote $\bV_h$.  We provide a detailed 
construction of the finite element space, in $d$ dimensions, in Appendix \ref{sec:construction-element}. The 
degrees of freedom are the vectorial velocities in the vertices of the macro elements and the velocity component in each face, $F = T \cap T'$,  in the direction pointing from the barycenter of triangle $T$ to that of $T'$.
We will apply $\bV_h \times Q_h$ in the finite element method for the Stokes' equations. \textcolor{black}{Observe that by
construction all functions $\bv_h \in \bV_h$ satisfy $\nabla \cdot
\bv_h \in Q_h$ and thus the divergence is constant on each macro element.}

We will need to bound quantities on $\Gamma$ using quantities in the bulk. To this end we recall the
following trace inequality \cite{WX19}, for all $v \in H^1(T)$ there holds 
\begin{equation}\label{eq:trace}
\|v\|_{\partial T} + \|v\|_{\Gamma \cap T} \lesssim h^{-\frac12} \|v\|_T + h^{\frac12} \|\nabla v\|_T.
\end{equation}

\subsection{Interpolants and approximation estimates}
We recall from \cite{BCH20} the following interpolant that commutes with the divergence operator.
For every $\bv \in [H^1(\Omega_{\mathcal{T}})]^d$ there exists $\pi_h \bv \in \bV_h$ such that
$\pi_h \bv(x_i) =i_h \bv(x_i)$ in the vertices $x_i$ of the macro elements, where $i_h$ denotes the Cl\'ement
interpolant on $\bW_h$, and for all $F \in \mathcal{F}$
\[
\int_F \pi_h \bv \cdot \bn_F ~\mbox{d}s = \int_F \bv \cdot \bn_F ~\mbox{d}s.
\]
Note that the interpolant $\pi_h \bv$ satisfies the approximation error estimate, for all $\bv \in H^l(\Omega_{\mathcal{T}})$,
\begin{equation}\label{eq:approx_glob}
\|\pi_h \bv - \bv\|_{\Omega_{\mathcal{T}}}
\lesssim h |\bu|_{[H^1(\Omega_{\mathcal{T}})]^d}, \quad h \|\nabla(\pi_h \bu - \bu)\|_{\Omega_{\mathcal{T}}} +  \|\pi_h \bu - \bu\|_{\Omega_{\mathcal{T}}}
\lesssim h^2 |\bu|_{[H^2(\Omega_{\mathcal{T}})]^d}.
\end{equation}
The proof of the existence of $\pi_h$ is identical to that of the
interpolant for the Bernardi-Raugel element \cite{BR85}, see also \cite{BCH20}. 
Since we are interested in unfitted finite element methods we recall from \cite{LNO21} a stable divergence free extension from $\Omega$ to $S$ that we denote by $\bu^e$. It was shown that $\bu^e$ satisfies the stability
\begin{equation}\label{eq:stab_divex}
\|\bu^e\|_{[H^{l}(S)]^d} \lesssim \|\bu\|_{[H^{l}(\Omega)]^d}, \quad l=1,2. 
\end{equation}
It follows that
\begin{equation}\label{eq:approx}
\begin{array}{rcl}
 \|\pi_h \bu - \bu\|_{\Omega}
\lesssim  h |\bu^e|_{[H^1(\Omega_{\mathcal{T}})]^d} &\leq& C_1 h \|\bu\|_{[H^1(\Omega)]^d}, \\[3mm]
 h \|\nabla(\pi_h \bu - \bu)\|_{\Omega} +  \|\pi_h \bu - \bu\|_{\Omega}
\lesssim  h^2 |\bu^e|_{[H^2(\Omega_{\mathcal{T}})]^d}&\leq& C_2 h^2 \|\bu\|_{[H^2(\Omega)]^d},\\[3mm]
\|h^{-\frac12}(\pi_h \bu - \bu)\|_{\Gamma}\lesssim  h |\bu^e|_{[H^2(\Omega_{\mathcal{T}})]^d} &\leq& C_3 h \|\bu\|_{[H^2(\Omega)]^d}
\end{array}
\end{equation}
where the last inequality follows by applying \eqref{eq:trace} on each element in $\mathcal{T}_\Gamma$ followed by the second inequality.
For the pressure variable, we will in principle extend by zero below and use standard element wise projections on piecewise constants in the bulk. We will also extend the multiplier variable $\blambda$ so that we can use approximation on the bulk mesh. For every $\bmu \in [H^{\frac12}(\Gamma)]^d$ there exists $\bmu^e \in [H^1(U_\delta)]^d$ with $\|\bmu^e\|_{[H^1(U_\delta)]^d} \lesssim \|\bmu\|_{[H^{\frac12}(\Gamma)]^d}$ (harmonic extension). Since $h<\delta$ and $\Omega_{\Gamma} \subset U_\delta$ we can define the approximation of $\bmu$ by $(\pi_\Gamma \bmu^e) \vert_{\Gamma} $,  where $\pi_\Gamma:L^2(\Omega_\Gamma) \rightarrow \bSigma_h$ is the standard $L^2$-projection on piecewise constant functions in $\bSigma_h$. Then there holds
\begin{equation}\label{eq:mulitplier_approx}
\|h^{\frac12}(\bmu^e - \pi_\Gamma \bmu^e)\|_{\mathcal{F}_i(\mathcal{T}_\Gamma)} + \|h^{\frac12}(\bmu - \pi_\Gamma \bmu^e)\|_\Gamma \lesssim h \|\bmu\|_{[H^{1/2}(\Gamma)]^d}.
\end{equation}
The inequality \eqref{eq:mulitplier_approx} also follows by applying \eqref{eq:trace} on each element in $\mathcal{T}_\Gamma$ followed by standard error estimates for $\pi_\Gamma$ and finally the stability of the extension.

Below we will make use of the local average of a function $y$ over some subdomain $X \subset  \Omega_{\mathcal{T}}$ we will use the following definition,
\[
\bar y^X := |X|^{-1} \int_X y ~\mbox{d}x, \mbox{ where } |X|:= \int_X ~\mbox{d}x.
\]
\section{Finite element discretization of the model problem}\label{sec:disc}
We consider the finite element spaces $\bV_h,\,Q_h$,  that were defined in the previous section,  and we will propose two different approaches to the weak imposition of boundary conditions. First a full Lagrange multiplier method where the normal stress on the boundary is introduced as an independent variable and then a Nitsche type method, where the pressure on the boundary is decoupled from the pressure in the bulk in the form of an additional multiplier controlling the normal component of the velocity. The tangential component of the velocity is then imposed using Nitsche's method. The former method, which stays very close to the continuous problem, has fewer terms that need to be evaluated and only one user defined parameter, while the latter has fewer unknowns, since only one multiplier field is added, but requires the evaluation of more boundary terms and three different stabilization parameters need to be set. Both methods can be analysed in the same fashion 
and we will give full detail only for the Lagrange multiplier method. The modifications necessary for the analysis of the Nitsche type method will be outlined in remarks.

\subsection{A Lagrange multiplier method}\label{sec:lagange}
In the first method simply discretize the form \eqref{eq:bilin_Brinkman} directly,  this means adding three unknowns for the normal stress on $\Gamma$. In this case only the Lagrange multiplier needs to 
be stabilized. The
finite element discretization of the problem  \eqref{eq:bilin_Brinkman} takes the form find $(\bu_h,p_h, \blambda_h) \in \bV_h \times Q_h \times \bSigma_h$ such
that
\begin{equation}\label{eq:FEM_brinkman}
A_h[(\bu_h,p_h,\blambda_h),(\bv_h,q_h,\bmu_h)] = l(\bv_h),\mbox{ for all } (\bv_h,q_h,\bmu_h) \in \bV_h\times
Q_h\times \bSigma_h.
\end{equation}
where the bilinear form is now defined by
\[
\begin{aligned}
A_h[(\bu_h,p_h,\blambda_h),(\bv_h,q_h,\bmu_h)] & := a(\bu_h,\bv) - b_h(p_h,\bv_h) + b_h(q_h,\bu_h) \\
&\qquad  + c(\blambda_h,\bv_h) - c(\bmu_h,\bu_h) + \gamma j(\blambda_h ,\bmu_h)
\end{aligned}
\]
where $\gamma>0$ is a stabilization parameter.
The modified form with subscript $h$ is defined by
\[
b_h(q_h,\bv_h) := (q_h,\nabla \cdot \bv_h)_{\Omega_\mathcal{T}}.
\]
We also introduce the vector valued stabilization term
\begin{equation}\label{eq:vec_stab}
j(\bv,\bw) := \sum_{F \in \mathcal{F}_i(\mathcal{T}_\Gamma)} \int_{F} h_T [\bv]\cdot [\bw]~\mbox{d}s.
\end{equation}
Here $[x]\vert_F$ denotes the jump of $x$ over face $F$. For future reference  we define $[x]\vert_F=x\vert_F$ on all $F$ in the mesh boundary, i.e. such that $F \cap \partial \Omega_{\mathcal{T}} = F$.

\subsection{A Nitsche method}
Here the boundary conditions on the velocities on $\Gamma$ are imposed using Nitsche's method. The velocity is enforced to be divergence free everywhere in $\Omega_\mathcal{T}$ through a bulk pressure variable and an additional unknown $\varrho$ is introduced for the pressure term on the boundary.
This unknown is a Lagrange multiplier imposing $\bu \cdot \bn_\Gamma = 0$. We formally define the discrete flux,
\[
\bsigma_h(\bv_h,\varpi_h) =   (\nabla \bv_h)^T \bn_\Gamma- \frac{\gamma_0}{2h} \bv_h  - \varpi_h \bn_\Gamma,
\]
where $\gamma_0 \in \mathbb{R}^+$ is a parameter that must be chosen large enough.
Replacing $\bu$, $p$, $\bv$ and $q$ in \eqref{eq:bilin_Brinkman} by $\bu_h$, $p_h$, $\bv_h$ and $q_h$ and using the above definition of $\bsigma_h(\bu_h,\varrho_h)$ and
 $\bsigma_h(\bv_h,-\varpi_h)$ instead of $\blambda$ and $\bmu$ we propose the formulation find $(\bu_h,p_h,\varrho_h) \in \bV_h\times
Q_h \times\Sigma_h$ such that:
\begin{equation}\label{eq:Nit}
A_{Nit}[(\bu_h,p_h,\varrho_h),(\bv_h,q_h,\varpi_h)] = l(\bv_h),\mbox{ for all } (\bv_h,q_h,\varpi_h) \in \bV_h\times
Q_h\times \Sigma_h
\end{equation}
where the bilinear form is defined by
\[
\begin{aligned}
&A_{Nit}[(\bu_h,p_h,\varrho_h),(\bv_h,q_h,\varpi_h)] := a(\bu_h,\bv) - b_h(p_h,\bv_h) + b_h(q_h,\bu_h) \\
&\qquad\qquad  - c(\bsigma_h(\bu_h,\varrho_h),\bv_h) - c(\bsigma_h(\bv_h,-\varpi_h),\bu_h) + \gamma_1 j(\varrho_h ,\varpi_h) + \gamma_2 g(\bu_h ,\bv_h).
\end{aligned}
\]
The stabilization terms are given by
\[
j(y_h ,z_h) := \sum_{F \in \mathcal{F}_i(\mathcal{T}_\Gamma)} \int_{F} h_T [y_h] [z_h]~\mbox{d}s
\]
and
\[
 g(\bu_h ,\bv_h):= \sum_{F \in \mathcal{F}(\mathcal{T}_\Gamma)} \int_{F\setminus \partial \Omega_{\mathcal{T}}} h_T [\nabla \bu_h]: [\nabla \bv_h]~\mbox{d}s.
\]
The stabilization operator $j$ is necessary to stabilize the Lagrange multiplier and $g$ is a ghost penalty term that is needed to make Nitsche's method robust independently of the mesh element intersection. Observe that compared to the method proposed in \cite{LNO21} the term $b_h$ acts on the whole mesh domain $\Omega_\mathcal{T}$. The pressure appearing in the stress approximation on the other hand is an independent variable and not coupled to the bulk pressure. Note that the above formulation includes three penalty parameters, one for the penalty term in Nitsche's method, $\gamma_0$, one for the ghost penalty $\gamma_1$ and one to stabilize the boundary pressure variable, $\gamma_2$. 

\begin{rem}\label{rem:ghost}(Ghost penalty)
The formulation \eqref{eq:FEM_brinkman} does not have any ghost penalty term for the velocities and since the pressure velocity coupling is integrated over the whole mesh domain none is needed for the pressure or the divergence of the velocities. Optimal error estimates are obtained below also in the absence of ghost penalty. However,  the linear system can be very ill-conditioned. This requires either a weakly consistent stabilizing term, whose only design condition is that \textcolor{black}{$g(\pi_h \bv,\pi_h \bv)^{\frac12} \lesssim h \|\bv\|_{[H^2(\Omega)]^d}$}, or an efficient preconditioner. The formulation \eqref{eq:Nit} on the other hand requires control of the gradient of $\bu_h$ over the whole mesh domain to counter the well known instability of Nitsche's method on unfitted meshes \cite{BCHLM15} leading to the need of a ghost penalty term in standard fashion. Since this term only acts on the velocity variable it does not perturb the divergence free property. Finally, in both cases the stabilization of the Lagrange multiplier acts on the bulk faces and therefore it both stabilizes the multiplier and provides stability for small cuts (see \cite{BH10}). 
\end{rem}

\begin{rem}(Average pressure)
Note that in the space $Q_h$ we have not imposed the zero average condition on the pressure. This is indeed not necessary, as we shall see below, since an artificial homogeneous Dirichlet condition is imposed on the pressure variable on $\Omega_{\mathcal{T}}$ (c.f. Remark \ref{rem:norm} and Theorem \ref{thm:infsup}).
This fixes the constant as a function of $\bu_h$ and the boundary multiplier $\blambda_h$ or $\rho_h$. This artificial boundary condition destroys the accuracy in the cut elements, but does not influence the accuracy on the interior elements. Therefore a globally accurate pressure with zero average can be constructed using post processing (c.f. equation \ref{eq:press_ext} and Corollary \ref{cor:cont_error}).
\end{rem}
\section{Stability and error analysis}\label{sec:estimate}
For the stability analysis it will be convenient to introduce the triple norm
\[
\tnorm{\bv,q,\bmu}:= \|\bv\|_{1,h}  + \|q\|_{0,\mathcal{T}_h} + \|h^{\frac12} \bmu\|_{\Gamma}  + j(\bmu,\bmu)^{\frac12},
\]
where 
\[
 \|\bv\|_{1,h} :=  \|\nabla \bv\|_{\Omega} + \|h^{-\frac12} \bv\|_{\Gamma}+ \|\nabla \cdot \bv\|_{\Omega_\mathcal{T}},
\]
and
\[
\|q\|_{0,\mathcal{T}_h}^2 := \|q - \bar q^{\Omega_\mathcal{T}}\|_{\Omega_\mathcal{T}}^2 + \|h^{\frac12} [q]\|_{\mathcal{F}(\mathcal{T}_h)}^2.
\]
\begin{rem}\label{rem:norm}
Observe that $\|\cdot\|_{0,\mathcal{T}_h}$ is a norm. This is easily seen since if $q = \bar q^{\Omega_\mathcal{T}}$, then $\|h^{\frac12} [q]\|_{\mathcal{F}(\mathcal{T}_h)} = \|h^{\frac12} \bar q^{\Omega_\mathcal{T}}\|_{\partial \Omega_{\mathcal{T}}}$ and hence $\|\cdot\|_{0,\mathcal{T}_h} = 0$ only if $\bar q^{\Omega_\mathcal{T}} = 0$. For the Lagrange multiplier method we have control of the full $H^1$-norm of the velocity only in $\Omega$. An extension to the mesh domain $\Omega_{\mathcal{T}}$ can be obtained in a standard fashion using ghost penalty. This is needed in the Nitsche method both in the error analysis and for the conditioning of the linear system, whereas in the Lagrange multiplier method the role of ghost penalty is only to enhance the conditioning of the system and could in principle be replaced by preconditioning.
\end{rem}

\subsection{Stability of the Lagrange multiplier}
The key difficulty in the present stability analysis is to obtain sufficient control of the Lagrange multiplier for the boundary condition without perturbing the pressure stability.  First we introduce four assumptions, in the spirit of \cite{Bur14}, that will ensure the stability of the Lagrange multiplier. Then we show that the assumptions of the abstract result are satisfied by the low order spaces and stabilization introduced in Section \ref{sec:FEMconst} and \ref{sec:disc}. 
We will also show that the Lagrange multiplier method gives similar control of the trace of $\bu$ as Nitsche's method. 
Note that we expect the proposed ideas to work also for higher order spaces provided the assumptions below are satisfied.
{\color{black}
\begin{assumption}\label{ass:velpress}
For all $q \in Q_h$ there exists $\bw_q \in \bV_h$ such that
\[
\|q\|_{0,\mathcal{T}_h}^2 \lesssim -b_h(q,\bw_q) \mbox{ and } \|\bw_q\|_{1,\Omega_{\mathcal{T}_h}} \lesssim \|q\|_{0,\mathcal{T}_h}.
\]
\end{assumption}
}
\begin{assumption}\label{ass:mult_stab}
There exists $\bSigma_H \subset \bSigma_h$, $H > h$, with the $L^2(\Gamma)$-projection $\pi_{H}:\bSigma_h \rightarrow  \bSigma_H$ and that there exists $C_\lambda, C_{H/h} > 0$, both depending on $H/h$ and the local mesh-geometry, such that for all $\bmu \in \bSigma_h$ there exists $\bvarphi \in \bV_h$ such that
\begin{equation}\label{eq:lam_infsup}
 \|h^{\frac12} \bmu \|^2_{\Gamma} \leq  c(\bmu, \bvarphi) + C_\lambda \|h^{\frac12} (\bmu- \pi_{H} \bmu)\|^2_\Gamma.
\end{equation}
and 
\begin{equation}\label{eq:var_stab}
\|\nabla \bvarphi\|_{\Omega_\mathcal{T}}^2 \leq C_{H/h} \|h^{\frac12} \bmu \|_\Gamma^2 \mbox{ and } \|h^{-\frac12} \bvarphi\|_{\Gamma}^2 \leq C_\lambda \|h^{\frac12} \bmu \|_\Gamma^2
\end{equation}
where $C_{H/h}$ can be made arbitrarily small by choosing $H/h$ large.
\end{assumption}
\begin{assumption}\label{ass:stab_op}
The stabilization operator $j:\bSigma_h \times \bSigma_h \rightarrow \mathbb{R}$ satisfies for all $\bmu \in \bSigma_h$
\begin{equation}\label{eq:lamvar_stab}
\|h^{\frac12} (\bmu - \pi_{H} \bmu)\|_{\textcolor{black}\Gamma}^2 \leq C_s j(\bmu,\bmu)
\end{equation}
where $C_s$ depends on $H/h$.
\end{assumption}
\begin{assumption}\label{ass:trac}
Assume that there exists $C_u > 0$, depending on $H/h$ and the local mesh-geometry, such that for all $\bv \in \bV_h$ there exists $\ssrf \in \bSigma_h$ such that
\begin{equation}\label{eq:vtrace_stab1}
\|h^{-\frac12} \bv\|_\Gamma^2 \leq -c(\ssrf,\bv) + C_u \|\nabla \bv\|_\Omega^2
\end{equation}
and 
\begin{equation}\label{eq:vtrace_stab2}
\|h^{\frac12} \ssrf\|_\Gamma^2  + j(\ssrf,\ssrf) \leq C_u \|h^{-\frac12} \bv\|^2_\Gamma.
\end{equation}
\end{assumption}
{\color{black}{
\begin{rem}
We prove below that the above assumptions are satisfied by the element introduced in Section \ref{sec:FEMconst}, but can also be shown to hold for the Raviart-Thomas or Brezzi-Douglas-Marini spaces if the $H^1$-nonconformity is handled appropriately. 
\end{rem}
\begin{prop}
The velocity pressure pair $\bV_h \times Q_h$ satisfies Assumption \ref{ass:velpress}.
\end{prop}
\begin{proof}
Under the assumption on the spaces (see \cite[Section 3.3]{BCH20} and Appendix) we can choose $\bw_1  \in \bV_h \cap [H^1_0(\Omega_\mathcal{T})]^d$ such that
\[
\nabla \cdot \bw_{1} =  \bar q^{\Omega_\mathcal{T}}-q \mbox{ in } \Omega_\mathcal{T} \mbox{ and } \|\nabla \bw_{1}\|_{\Omega_{\mathcal{T}_h}} \lesssim \|q - \bar q^{\Omega_\mathcal{T}}\|_{\Omega_\mathcal{T}},
\]
 to get
$$
-b_h(q,\bw_1) =   \|q - \bar q^{\Omega_\mathcal{T}}\|_{\Omega_{\mathcal{T}}}^2.
$$
We obtain the bound $\|h^{-\frac12} \bw_1\|_{\Gamma} \lesssim \|\nabla \bw_1\|_{\Omega_\Gamma}$ using the trace inequality \eqref{eq:trace} followed by the scaled Poincar\'e inequality (see \cite[Appendix]{BHL18}),
\[
h^{-\frac12} \|\bw_1\|_{\Omega_\Gamma} \lesssim \|\nabla \bw_1\|_{\Omega_\Gamma}.
\]

Using integration by parts we also have for all $q \in Q_h$,
\[
b_h(q,\bw) = \sum_{T} \int_{\partial T} q \bw \cdot \bn_{\partial T} ~\mbox{d}s.
\]

By the construction of $\bV_h$, we can set the vertex degrees of freedom to zero, and using only the vector degrees of freedom on the faces, we can construct $\bw_{2} \in \bV_h$ so that $\int_{\partial T} q \bw_{ 2} \cdot \bn_{\partial T}~\mbox{d}s = -\int_{\partial T} h q [q] ~\mbox{d}s$ and $\|\bw_{2} \cdot \bn\|_{\mathcal{F}(\mathcal{T}_h)} \lesssim \|h [q]\|_{\mathcal{F}(\mathcal{T}_h)}$. Summing over all the elements it follows that
\[
-b_h(q,\bw_{2}) =  \|h^{\frac12} [q]\|_{\mathcal{F}(\mathcal{T}_h)}^2.
\]
Note that by applying an inverse inequality, an inverse trace inequality and by the choice of face degrees of freedom we have
\[
\|\nabla \bw_{2}\|_{\Omega_\mathcal{T}} \lesssim \|h^{-1} \bw_{2}\|_{\Omega_\mathcal{T}}
 \lesssim \|h^{-\frac12} \bw_{2} \cdot \bn\|_{\mathcal{F}(\mathcal{T}_h)} \lesssim  \|h^{\frac12} [q]\|_{\mathcal{F}(\mathcal{T}_h)}.
 \]
We conclude by taking $\bw_{q} = \bw_1 + \bw_2$.
\end{proof}
}}
\begin{prop}
The spaces $\bV_h$ and $\bSigma_h$ introduced in Section \ref{sec:FEMconst} and the stabilization operator $j(\cdot,\cdot)$ defined by equation \eqref{eq:vec_stab} satisfy the Assumptions \ref{ass:mult_stab}-\ref{ass:trac}.
\end{prop}
\begin{proof}
We construct $\bSigma_H$ by
agglomerating the elements of $\mathcal{T}_{\Gamma}$ in $N$ disjoint, boundary elements of radius $O(H) \sim (|\Gamma|/N)^{\frac{1}{d-1}}$ with $H \sim M h$ for some $M \ge 1$. Then we extend every surface patch into the bulk a distance $O(H)$ to create bulk patches $\{P_i\}_{i=1}^N$, see Figure \ref{fig:patch}.  Let $\Gamma_i:=\Gamma \cap P_i$, by construction $\Gamma = \cup_{i=1}^N \Gamma_i$ and define $\pi_H \bmu\vert_{\Omega_\Gamma \cap P_i}:=\bar \bmu^{\Gamma_i}$. Observe that by varying $h$ and $N$ we can use $M$ as a free parameter. We assume that $M$ is large enough so that the largest ball $B_i$ with center $\bx_i$ on $\Gamma_i$, radius $r_i = O(H)$ and $B_i\cap \Omega_\mathcal{T}  \subset P_i$, contains at least one vertex $\by_i$ of $\mathcal{T}_h$, such that $|\by_i - \bx_i| \leq r_i/2$ for all $i$.  With these preparations we define
\[
\bSigma_H := \{\bmu_H \in L^2(\Omega_{\Gamma}) : \bmu_H\vert_{\Omega_\Gamma \cap P_i} \in \mathbb{P}_0(\Omega_\Gamma \cap P_i), 1\leq i \leq N \}.
\]
\begin{figure}
\begin{center}
\setlength{\unitlength}{1.0cm}
\begin{picture}(10,5)
\includegraphics[scale=0.8]{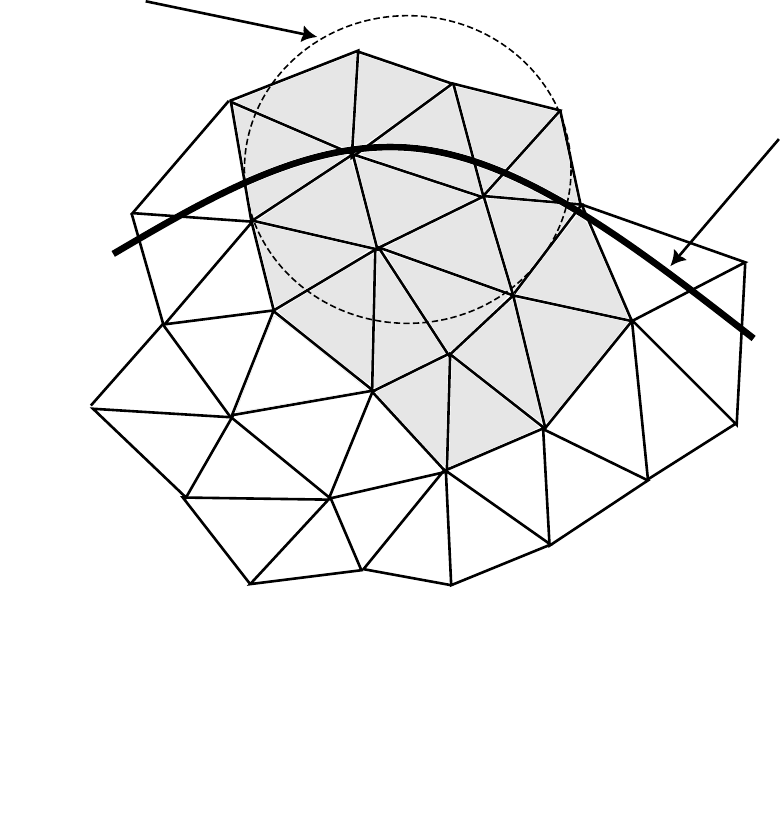}
\thicklines
\multiput(-1.2,2.)(0, 0.4){12}{\line(0,1){0.2}}
\multiput(-4.5,2.)(0, 0.4){12}{\line(0,1){0.2}}
\multiput(-5.2,6.3)(0.4,0){14}{\line(1,0){0.2}}
\multiput(-5.2,3)(0.4,0){14}{\line(1,0){0.2}}

\put(-5.7,6.7){$B_i$}
\put(-0.1,5.7){$\Gamma$}

\put(-0.2,4.55){{{$\left.\begin{array}{l}  \\ \\ \\ \\ \\ \\ \\ [4mm]\end{array}
      \right\}$}}}
\put(0.6,4.55){$2 H$}
\put(-2.9,4.6){$P_i$}
\put(-4.5,1.8){{\Large{$\underbrace{\mbox{  \hspace{2.9cm} }}_{2 H}$}}}
\end{picture}
\end{center}
\vspace{-1.5cm}
\caption{Illustration of a patch $P_i$ in two space dimensions. The patch consist of the shaded triangles.}\label{fig:patch}
\label{fig:lemSh}
\end{figure}
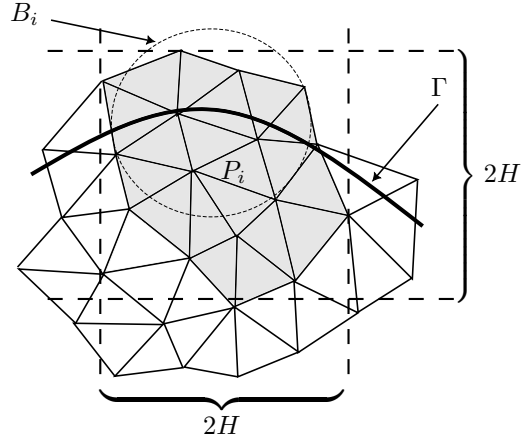

{Assumption  \ref{ass:stab_op}} now follows from the following discrete interpolation result, for all $\bmu \in \bSigma_h$ there holds
\begin{equation} \label{eq:disc_interp}
\sum_{i=1}^N h \|\bmu - \bar \bmu^{\Gamma_i}\|_{\Gamma_i}^2 \leq C_s j(\bmu,\bmu),
\end{equation}
where $C_s$ is a constant that depends only on $M$. For completeness we give a proof in the appendix. 

Next we consider Assumption \ref{ass:mult_stab}. To verify (\ref{eq:var_stab}) we prove in the appendix that by construction,  for 
$M$ large enough, there exist functions 
$\bvarphi_i \in \bV_h$ such that $\bar \bvarphi_i^{\Gamma_i} = 2 h \bar \bmu^{\Gamma_i}$ and 
\begin{equation}\label{eq:fcn_stab}
\|\nabla \bvarphi_i\|_{\Gamma_i} +H^{-\frac12} \|\nabla \bvarphi_i\|_{P_i}\leq C M^{-1} \|\bar \bmu^{\Gamma_i}\|_{\Gamma_i}.
\end{equation}
It follows that if $\bvarphi = \sum_{i=1}^N \bvarphi_i$ then
\[
\|\nabla \bvarphi_i\|_{\Omega_{\mathcal{T}}} \leq C H^{\frac12}/h^{\frac12} M^{-1} \|h^{\frac12}\pi_H \bmu\|_{\Gamma} \leq C M^{-\frac12} \|h^{\frac12} \bmu\|_{\Gamma}
\]
where we used the stability of the $L^2$-projection in the last step. This shows the left inequality of \eqref{eq:var_stab}, with $C_{H/h} = O(M^{-1})$.
Noting that by Poincar\'e's inequality $\|\bvarphi_i\|_{\Gamma_i} \leq C H \|\nabla_\Gamma \bvarphi_i\|_{\Gamma_i }$, where $\nabla_\Gamma$ denotes the tangential gradient on $\Gamma$,
\begin{equation}\label{eq:vargamma}
\|h^{-\frac12} \bvarphi\|_{\Gamma}  \leq C M \left(\sum_{i=1}^N h \|\nabla_\Gamma \bvarphi_i\|_{\Gamma_i }^2 \right)^{\frac12} \leq C  \|h^{\frac12} \pi_H \bmu\|_{\Gamma}.
\end{equation}
This proves the right inequality of \eqref{eq:var_stab} after application of the stability of the $L^2$-projection. To prove \eqref{eq:lam_infsup} note that by construction
\[
c(\bmu, \bvarphi) = c(\bmu - \bar \bmu, \bvarphi) + 2 \|h^{\frac12} \pi_H \bmu\|_{\Gamma}^2 \ge - \|h^{\frac12}(\bmu - \pi_H \bmu)\|_{\Gamma} \|h^{-\frac12} \bvarphi\|_{\Gamma} + 2 \|h^{\frac12} \pi_H \bmu\|_{\Gamma}^2.
\]
Estimating the right hand side from below using \eqref{eq:vargamma} followed by the arithmetic-geometric inequality we obtain
\[
c(\bmu, \bvarphi) \ge \|h^{\frac12} \pi_H \bmu\|_{\Gamma}^2 - C \|h^{\frac12}(\bmu - \pi_H \bmu)\|_{\Gamma}^2 .
\]
and using that $\|h^{\frac12} \bmu\|_{\Gamma}^2 = \|h^{\frac12} \pi_H \bmu\|_{\Gamma}^2+ \|h^{\frac12} (\bmu - \pi_H \bmu) \|_{\Gamma}^2$ concludes the verification of \eqref{eq:lam_infsup}.

Finally,  to verify Assumption \ref{ass:trac} we need a test function that recovers control of $\bv\vert_{\Gamma}$. To this end let $\ssrf= -2 h^{-1} \pi_H \bv$ and observe that by construction we have
\[
-c(\ssrf, \bv) =  2 \|h^{-\frac12} \pi_H \bv\|^2 \ge  \|h^{-\frac12}  \bv\|_\Gamma^2 - C_M \|\nabla \bv\|_{\Omega}^2.
\]
In the last inequality we used the bounds
\[
\|h^{-\frac12} \bv\|^2_{\Gamma} \leq \|h^{-\frac12} (\bv - \pi_H \bv)\|^2_{\Gamma} + \|h^{-\frac12}  \pi_H \bv\|^2_{\Gamma} \leq C_M \|\nabla \bv\|^2_{\Omega} + \|h^{-\frac12}  \pi_H \bv\|^2_{\Gamma},
\]
where we used trace inequality on each patch $P_i\cap \Omega$, $\|\bv\|_{\Gamma_i} \lesssim H^{-\frac12} \|\bv\|_{P_i\cap \Omega} + H^{\frac12}\|\nabla \bv \|_{P_i\cap \Omega}$  followed by a Poincar\'e inequality $\|\bv - \bar \bv\|_{P_i\cap \Omega} \lesssim H \|\nabla \bv \|_{P_i\cap \Omega}$ to get
\[
\|h^{-\frac12} (\bv - \bar \bv)\|_{\Gamma_i} \lesssim M^{-\frac12} h^{-1} \|\bv - \bar \bv\|_{P_i\cap \Omega}  + M^{\frac12} \|\nabla \bv \|_{P_i\cap \Omega} \leq C M^{\frac12} \|\nabla \bv \|_{P_i \cap \Omega}.
\]
It follows that $C_u = O(M^{\frac12})$.
To verify \eqref{eq:vtrace_stab2} we see that by the definition of $\ssrf$ and the stability of the $L^2$-projection, $\|\pi_H \bv\|_{\Gamma_i} \leq \|\bv\|_{\Gamma_i}$,
\[
\|h^{\frac12} \ssrf\|_{\Gamma}^2 = \sum_{i=1}^N \| h^{\frac12} 2 h^{-1} \pi_H \bv\|^2_{\Gamma_i} \lesssim  \| h^{\frac12} h^{-1} \bv\|^2_{\Gamma}.
\]
For the second term using trace and inverse trace inequalities and the stability of the $L^2$-projection we obtain
\[
j(\ssrf,\ssrf) \lesssim \|\ssrf\|_{\Omega_\Gamma}^2 \leq \|h^{\frac12} \ssrf\|_{\Gamma}^2
\]
we conclude as before.
\end{proof}
\subsection{The global infsup condition}
Here we prove a global infsup condition under the Assumptions \ref{ass:mult_stab}-\ref{ass:trac} and that the velocity-pressure spaces are inf-sup stable with $\nabla \cdot \bV_h \subset Q_h$. For the pressure stability we use the particular properties of  $\bV_h \times Q_h$, but the analysis is straightforward to extend to other inf-sup stable velocity-pressure pairs with similar properties.
\begin{thm}\label{thm:infsup}
Let the spaces $\bV_h,\,\bSigma_h$ satisfy Assumptions \ref{ass:mult_stab}-\ref{ass:stab_op}.
There exists $h_0>0$, $\alpha>0$ such that for all $\bv,q,\bmu \in \bV_h\times Q_h\times \bSigma_h$, $h < h_0$ there exists $\bw,y,\srf \in \bV_h\times Q_h\times \bSigma_h$ such that
\begin{equation}\label{eq:Astab_1}
\alpha \tnorm{\bv,q,\bmu}^2 \leq A_h[(\bv,q,\bmu),(\bw,y,\srf )] 
\end{equation}
and
\begin{equation}\label{eq:tnorm_stab}
\tnorm{\bw,y,\srf} \lesssim \tnorm{\bv,q,\bmu}.
\end{equation}
\end{thm}
\begin{proof}
The proof proceeds by choosing different test functions to get control of the various terms in the norm $\tnorm{\cdot,\cdot,\cdot}$. 
\begin{enumerate}
\item \underline{Control of the $H^1$-norm of velocities, the divergence on $\Omega_{\mathcal{T}}$ and stabilization.}
First note that since $\nabla \cdot \bV_h \subseteq  Q_h$  we have
\[
A_h[(\bv,q,\bmu),(\bv,q +  \nabla \cdot \bv,\bmu )] = \|\nabla \bv\|_{\Omega}^2 + \gamma j(\bmu,\bmu) +  \|\nabla \cdot \bv\|_{\Omega_{\mathcal{T}}}^2.
\]
\textcolor{black}{
\item \underline{Control of the $L^2$-norm of the pressure on $\Omega_{\mathcal{T}}$.} Using Assumption \ref{ass:velpress} we can choose $\bw = 2 \bw_{q} \in \bV_h $
 to get
\[
A_h[(\bv,q,\bmu),(\bw,0,0)] \ge  \|q\|_{0,\Omega_{\mathcal{T}}}^2 - C  (\|\nabla \bv\|_{\Omega}^2 + \|h^{\frac12} \bmu\|_{\Gamma}^2).
\]
Here we used that 
\begin{align*}
|a(\bv,\bw)+c(\bmu,\bw)|  &\leq (\|\nabla \bv\|_{\Omega} + \|h^{\frac12} \bmu\|_{\Gamma})(\|\nabla \bw\|_{\Omega} + \|h^{-\frac12} \bw\|_{\Gamma})
\\ 
&  \lesssim (\|\nabla \bv\|_{\Omega} + \|h^{\frac12} \bmu\|_{\Gamma})\|q\|_{0,\Omega_{\mathcal{T}}} \\
& \leq \|q\|_{0,\Omega_{\mathcal{T}}}^2 + C (\|\nabla \bv\|_{\Omega}^2 + \|h^{\frac12} \bmu\|_{\Gamma}^2).
\end{align*}
}
\item \underline{Control of the Lagrange multiplier for the boundary condition.}
Turning now to the Lagrange multiplier we use Assumption \ref{ass:mult_stab} to choose $\bvarphi \in \bV_h$ such that
\[
\|h^{\frac12} \bmu \|^2_{\Gamma} + a(\bv,\bvarphi) - b_h(q,\bvarphi)  \leq A_h[(\bv,q,\bmu),(\bvarphi,0,0)] +C_\lambda \|h^{\frac12} (\bmu- \pi_{H} \bmu)\|^2_\Gamma
\]
Using the stability of $\bvarphi$, \eqref{eq:lamvar_stab} it is straightforward to show that,
\[
|a(\bv,\bvarphi) - b_h(q,\bvarphi)| \leq  (\|\nabla \bv\|_\Omega + \|q\|_{\Omega_{\mathcal{T}}})\|\nabla \bvarphi\|_{\Omega_{\mathcal{T}}} \leq (\|\nabla \bv\|_\Omega + \|q\|_{\Omega_{\mathcal{T}}}) C_{H/h}^{\frac12} \|h^{\frac12} \bmu\|_{\Gamma}
\]
and hence applying also Assumption \ref{ass:stab_op} we have
\[
\frac34  \|h^{\frac12}  \bmu\|_{\Gamma}^2 - C_{H/h} (\|\nabla \bv\|_\Omega^2 + \|q\|^2_{0,\mathcal{T}_h}) - C_\lambda C_s j(\bmu,\bmu) \leq A_h[(\bv,q,\bmu),(\bvarphi,0,0)].
\]
\item \underline{Control of $\bv$ on the boundary.} We use the Assumption \ref{ass:trac} to choose $\ssrf \in \bSigma_h$ such that
\[
\|h^{-\frac12} \bv \|^2_{\Gamma} + \gamma j(\bmu,\ssrf) - C_u \|\nabla \bv\|_\Omega^2  \leq A_h[(\bv,q,\bmu),(0,0,\ssrf)].
\]
By the Cauchy-Schwarz inequality and \eqref{eq:vtrace_stab2} it follows that
\[
\frac34 \|h^{-\frac12} \bv \|^2_{\Gamma} - C_u \gamma^2 j(\bmu,\bmu) - C_u \|\nabla \bv\|_\Omega^2 \leq A_h[(\bv,q,\bmu),(0,0,\ssrf)].
\]
\item \underline{Proof of the bound \eqref{eq:Astab_1}.} 
To get the first bound we take $\bw = \bv + \epsilon_1 \bw_q + \epsilon_2 \bvarphi$ and $y=q +  \nabla \cdot \bv$, $\srf = \bmu + \epsilon_3 \ssrf$  leading to
\begin{align*}
&A_h[(\bv,q,\bmu),(\bv + \epsilon_1 \bw_q + \epsilon_2 \bvarphi,q +  \nabla \cdot \bv,\bmu+\epsilon_3 \ssrf)] 
\\
&\qquad
\ge (1 - \epsilon_1 C_p - \epsilon_2 C_{H/h} -  C_u \epsilon_3)) \|\nabla \bv\|_{\Omega}^2
 + (\epsilon_1   -  \epsilon_2 C_{H/h}) \|q\|_{0,\mathcal{T}_h}^2\\
&\qquad \qquad +  \|\nabla \cdot \bv\|_{\Omega}^2   + (\gamma - \epsilon_2 C_\lambda C_s  - \epsilon_3 C_u \gamma^2) j(\bmu,\bmu) 
 \\
 &\qquad \qquad   + 
 (3 \epsilon_2 /4  - \epsilon_1 C_p ) \|h^{\frac12} \bmu\|_{\Gamma}^2 +  3 \epsilon_3/4 \|h^{-\frac12} \bv\|^2_\Gamma.
\end{align*}
First fix $\epsilon_1 = \epsilon_2/(4 C_p)$ so that 
\[
 (3 \epsilon_2 /4  - \epsilon_1 C_p ) = \epsilon_2/2 \mbox{ and } 1 - \epsilon_1 C_p - \epsilon_2 C_{H/h} -  C_u \epsilon_3 = 1- \epsilon_2(1/4+C_{H/h}) - C_u \epsilon_3.
\]
Now fix $H/h$ sufficiently large so that $C_{H/h}< \min((8 C_p)^{-1}, 1/4)$ then
\[
(\epsilon_1 - \epsilon_2 C_{H/h}) \ge \epsilon_2 (8 C_p)^{-1}
\]
and
\[
1- \epsilon_2(1/4+C_{H/h}) - C_u \epsilon_3 \ge 1- \epsilon_2/2 - C_u \epsilon_3.
\]
Since we have fixed $H/h$ to make $C_{H/h}$ small enough, $C_s$ and $C_u$ are also fixed and we can conclude by choosing $\epsilon_2 = \min(\gamma/(4 C_\lambda C_s), 1/2)$ and $\epsilon_3 = 1/(4 C_u) \min(1,1/\gamma)$ to obtain
\[
1- \epsilon_2/2 - C_u \epsilon_3 \ge 1/2
\] 
and
\[
(\gamma - \epsilon_2 C_\lambda C_s  - \epsilon_3 C_u \gamma^2)\ge \gamma/2.
\]
The first equality then holds with $\alpha = \tfrac14 \min_i \alpha_i$
where
\[
\alpha_1 = 1, \, \alpha_2 = 2/\gamma,\, \alpha_3 = \gamma/(2 C_\lambda C_s),\, \alpha_3 = 3/(4 C_u)\min(1,1/\gamma)), \, \alpha_4 = \tfrac{1}{2 C_p} \min(1, (C_\lambda C_s)^{-1}).
\]
\item \underline{Proof of the stability \eqref{eq:tnorm_stab}.} We want to establish that
\[
\tnorm{\bv + \epsilon_1 \bw_q + \epsilon_2 \bvarphi,q+ \nabla \cdot \bv,\bmu+\epsilon_3 \ssrf} \lesssim \tnorm{\bv,q,\bmu}.
\]
Using the triangle inequality we have
\[
\begin{aligned}
\tnorm{\bv + \epsilon_1 \bw_q + \epsilon_2 \bvarphi,q + \nabla \cdot \bv,\bmu+\epsilon_3 \ssrf} & \lesssim \tnorm{\bv,q,\bmu} + \epsilon_1 \| \bw_q\|_{1,h} +\epsilon_2 \|\bvarphi\|_{1,h}\\
& \qquad + \|\nabla \cdot \bv\|_{\mathcal{T}_h}+ \epsilon_3 \|h^{\frac12} \ssrf\|_{\Gamma} + \epsilon_3 j(\ssrf,\ssrf)^{\frac12}.
\end{aligned}
\]
We consider the terms in the right hand side one by one
\[
\| \bw_q\|_{1,h} \lesssim \|\nabla \bw_q\|_{\Omega_{\mathcal{T}}} + \|h^{-\frac12} \bw_q\|_{\Gamma}
\lesssim \|\nabla \bw_q\|_{\Omega_{\mathcal{T}}}  + \|h^{-\frac12} \bw_{[q]}\|_{\Gamma}.
\]
Now observe that by definition $\|\nabla \bw_q\|_{\Omega_{\mathcal{T}}}  \lesssim \|q -\bar q^{\Omega_\mathcal{T}} \|_{\Omega_{\mathcal{T}}} \lesssim  \|q\|_{0,\mathcal{T}_h}$ and for the boundary term we have using trace inequalities and the definition of $\bw_{[q]}$,
\[
\|h^{-\frac12} \bw_{[q]}\|_{\Gamma} \lesssim \|\nabla \bw_{[q]}\|_{\Omega_{\mathcal{T}}} + \|h^{-1} \bw_{[q]}\|_{\mathcal{T}_\Gamma} \lesssim \|h^{\frac12} [q]\|_{\mathcal{F}(\mathcal{T}_\Gamma)} \lesssim \|q\|_{0,\mathcal{T}_h}.
\]
The divergence is controlled using the stability of the average and trace inequalities,
\[
\|\nabla \cdot \bv\|_{0,\mathcal{T}_h}\lesssim \|\nabla \cdot \bv\|_{\Omega_\mathcal{T}},
\]
%
Then observe that by \eqref{eq:var_stab} as seen above,
\[
\|\bvarphi\|_{1,h} \lesssim \|h^{\frac12} \bmu\|_{\Gamma}
\]
and by \eqref{eq:vtrace_stab2} 
\[
\|h^{\frac12} \ssrf\|_{\Gamma} + j(\ssrf,\ssrf)^{\frac12} \lesssim \|h^{-\frac12} \bv\|_{\Gamma}.
\]
The claim follows by collecting the bounds.
\end{enumerate}
\end{proof}

\begin{rem}
Observe that there are no constraints on $\gamma$, but the constant $\alpha$ degenerates if $\gamma$ becomes too small or too large. The key constraint is that the mesh must be sufficiently fine so that it is possible to satisfy both the divergence free condition and the boundary condition at the same time. \textcolor{black}{
Essentially, making $H/h$ larger reduces the strength of the imposition of the  boundary condition on the divergence free space.
}
Since the method is unfitted there is no reason both the divergence free condition and the Dirichlet condition can be imposed exactly without locking. It follows that although the approximation satisfies the divergence free condition pointwise, the associated pointwise zero flow condition on the boundary is only be satisfied asymptotically.

\end{rem}

\begin{rem}
A similar result can be proven for the Nitsche method \eqref{eq:Nit}, there are two main differences. First we note that for $\gamma_0$ large enough we have using standard arguments
\[
 \|\nabla \bv\|_{\Omega_{\mathcal{T}}}^2+ \|h^{-\frac12} \bv\|_{\Gamma}^2 + \gamma_1 j(\mu,\mu) +  \|\nabla \cdot \bv\|_{\Omega_{\mathcal{T}}}^2\lesssim A_{Nit}[(\bv,q,\mu),(\bv,q +  \nabla \cdot \bv,\mu )].
\]
This time we see that coercivity holds over all of $\Omega_{\mathcal{T}}$ thanks to the ghost penalty term. The second difference is that since the Lagrange multiplier is a scalar, it must be paired with a scalar function for stability, the form however is $(\mu_h \bn_\Gamma, \bv_h)_\Gamma$. To see how this can be handled it is enough to consider one $\Gamma_i$. We add and subtract $\bar \bn_{\Gamma}^{\Gamma_i}$ to obtain
\[
(\mu \bn_\Gamma, \bv)_{\Gamma_i} = (\mu (\bn_\Gamma - \bar \bn_{\Gamma}^{\Gamma_i}),\bv)_{\Gamma_i} + (\mu  \bar \bn_{\Gamma}^{\Gamma_i} , \bv_h)_{\Gamma_i}.
\]
We may now choose $\bvarphi$ such that for  $\overline{\bvarphi}^{\Gamma_i}_i = h \bar \mu^{\Gamma_i} \bar \bn_{\Gamma}^{\Gamma_i}$ and recall that $\|\bn_{\Gamma} - \bar \bn_{\Gamma}^{\Gamma_i}\|_{L^\infty(\Gamma_i)} \leq C_\Gamma h$ to obtain
\[
(\mu \bn_\Gamma, \bvarphi_i)_{\Gamma_i} \ge (1 - C_\Gamma h) \|h^{\frac12} \bar \mu^{\Gamma_i}\|_{\Gamma_i}^2 - C_\Gamma h j(\mu,\mu).
\]
Using this test function yields control of the boundary pressure variable $\mu$ as before.
The pressure analysis carries over verbatim and of course no action needs to be taken for the control of $\bv$ on $\Gamma$, which is controlled through the Nitsche penalty term.
\end{rem}
\subsection{Error analysis}
We now prove an error bound.
Let $\bu^e$ denote a stable divergence free extension of $\bu$ (see \cite{LNO21}). Let $\pi_h \bu^e \in \bV_h$ be the interpolant introduced above which preserves the divergence free property.
Let $\pi_\Gamma \blambda^e$ be the $L^2$-projection of $\blambda^e$ on $\bSigma_h$. Finally we define $\pi_{\mathcal{T}}$ to be the local $L^2$-projection on piecewise constants in every element in $\mathcal{T}_h$, defined by
\begin{equation}\label{eq:piT}
(\pi_{\mathcal{T}} p,q)_{\Omega_\mathcal{T}} = (p,q)_{\Omega}, \quad \forall q \in Q_h.
\end{equation}
Note that this projection does not produce an accurate approximation of $p$ in the cells in $\mathcal{T}_{\Gamma}$. This is of no importance below, because the role of the projection is to make the pressure disappear, not to approximate it. An optimal global pressure is then obtained using post processing. Note however that by definition $\overline{\pi_{\mathcal{T}} p}^{\Omega_\mathcal{T}} = \bar p^{\Omega} = 0$.
\begin{thm}\label{thm:disc_error}(Error estimate for the discrete error)
Under our assumptions on $\Omega$ the exact solution satisfies $\bu \in [H^2(\Omega)]^d$ and $\blambda \in [H^{\frac12}(\Gamma)]^d$. There holds
\[
\tnorm{\pi_h \bu - \bu_h, \pi_{\mathcal{T}} p - p_h, \pi_\Gamma \blambda - \blambda_h} \lesssim h (\|\bu\|_{[H^2(\Omega)]^d} + \|\blambda\|_{[H^{\frac12}(\Gamma)]^d}).
\]
\end{thm}
\begin{proof}
Let $\be_h = \pi_h \bu^e - \bu_h$, $\bfeta_h = \pi_\Gamma \blambda^e - \blambda_h$ and $\omega_h = \pi_{\mathcal{T}} p - p_h$. Using the inf-sup stability of Theorem \ref{thm:infsup} we have
\[
\begin{aligned}
\alpha \tnorm{\be_h,\omega_h,\bfeta_h}^2 & \leq A_h[(\be_h,\omega_h,\bfeta_h),(\bw,y,\srf )] 
\\
&=  A_h[(\pi_h \bu^e,\pi_{\mathcal{T}} p,\pi_\Gamma \blambda),(\bw,y,\srf )] - 
l(\bw) 
\\
&=  A_h[(\pi_h \bu^e,\pi_{\mathcal{T}} p,\pi_\Gamma \blambda),(\bw,y,\srf )] - A[(\bu,p,\blambda),(\bw,y,\srf )]
\\
 & = a(\pi_h \bu^e - \bu, \bw) - b_h(\pi_{\mathcal{T}_h} p,\bw) + b(p, \bw) + c(\bw,\pi_\Gamma \blambda^e - \blambda)
 \\
 &\qquad  -c(\pi_h \bu^e - \bu,\srf) + \gamma j(\pi_\Gamma \blambda^e ,\srf).
\end{aligned}
\]
By the definition  \eqref{eq:piT}  of $\pi_{\mathcal{T}}$ we see that, since $\nabla \cdot \bw \in Q_h$, 
\[
b_h(\pi_{\mathcal{T}} p,\bw) - b(p, \bw) = (\pi_{\mathcal{T}} p, \nabla\cdot  \bw)_{\Omega_{\mathcal{T}}} - (p, \nabla \cdot \bw)_{\Omega} = 0.
\]
We now bound the remaining terms of the right hand side using the Cauchy-Schwarz inequality and the approximation bounds \eqref{eq:approx}, \eqref{eq:mulitplier_approx}. First consider the weak Laplacian, 
\[
a(\pi_h \bu^e - \bu, \bw) \lesssim \|\nabla (\pi_h \bu^e - \bu)\|_{\Omega} \tnorm{\bw,0,0} \lesssim  h \|\bu\|_{[H^2(\Omega)]^d} \tnorm{\bw,0,0}.
\]
For the multiplier terms we have
\[
c(\bw, \pi_\Gamma \blambda^e - \blambda) \lesssim \|h^{\frac12}(\pi_\Gamma \blambda^e - \blambda)\|_{\Gamma} \|h^{-\frac12} \bw\|_{\Gamma} \lesssim  h \|\blambda\|_{[H^{\frac12}(\Gamma)]^d}\tnorm{\bw,0,0},
\]
\[
c(\pi_h \bu^e - \bu,\srf) \lesssim \|h^{-\frac12} (\pi_h \bu^e - \bu)\|_{\Gamma} \|h^{\frac12} \srf\|_\Gamma \lesssim h \|\bu\|_{[H^2(\Omega)]^d} \tnorm{0,0,\srf},
\]
and since $\blambda^e\vert_{\Omega_\Gamma} \in [H^1(\Omega_\Gamma)]^d$,
\[
j(\pi_\Gamma \blambda^e,\srf)=j(\pi_\Gamma \blambda^e - \blambda^e,\srf) \lesssim j(\pi_\Gamma \blambda^e - \blambda^e,\pi_\Gamma \blambda^e - \blambda^e)^{\frac12} j(\srf,\srf)^{\frac12} \lesssim  h \|\blambda\|_{[H^{\frac12}(\Gamma)]^d}\tnorm{0,0,\srf}.
\]
Collecting the above bounds we obtain
\[
\alpha \tnorm{\be_h,\omega_h,\bfeta_h}^2 \lesssim h(\|\bu\|_{[H^2(\Omega)]^d}+\|\blambda\|_{[H^{\frac12}(\Gamma)]^d})\tnorm{\bw,0,\srf}.
\]
The claim follows using the second inequality of Theorem \ref{thm:infsup}.
\end{proof}
Observe that this does only lead to an optimal error estimate uniformly in the domain for the continuous variables $\bu$ and $\blambda$. The pressure approximation deteriorates on the boundary due to the nonconsistency of $\pi_{\mathcal{T}} p$.  However,  using local postprocessing we obtain an approximation of the pressure that has optimal order of approximation. To this end we define the canonical extension of a polynomial $x$ defined on a simplex $T$ to $\mathbb{R}^d$ by $E_T x$. For all triangles $T \in \mathcal{T}_\Gamma$ we introduce a mapping $S_h$ that to each element in $\mathcal{T}_\Gamma$ associates an element in $\mathcal{T}_I$ a distance $O(h)$ away. For a detailed discussion of the mapping $S_h$ we refer to \cite{BHL22}. For every function $q_h \in Q_h$ we then define
\begin{equation}\label{eq:press_ext}
q_h^e := \left\{
\begin{array}{l}
q_h - \bar q_h^{\Omega_{\mathcal{T}}} \mbox{ for } x \in \Omega_I \\
E_{S_h(T)} q_h - \bar q_h^{\Omega_{\mathcal{T}}} \mbox{ for } x \in T \subset \Omega_\Gamma
\end{array}\right..
\end{equation}
Now for each pair $(T,S_h(T))$ let $B_{T,S_h} \subset \IR^d$ denote a ball with radius $O(h)$ such that $T\cup S_h(T) \subset B_{T,S_h}$. We let $\pi_B$ denote the $L^2$-projection on constant functions on $B_{T,S_h}$ and recall that for all $q \in H^1(B_{T,S_h})$ there holds $\|q - \pi_B q\|_{B_{T,S_h}} \lesssim h |q|_{H^1(B_{T,S_h})}$. Also observe that under the regularity assumptions of the mesh the overlap of the balls is finite. We also introduce the extension of the continuous pressure $p^e \in H^1(S)$ such that $p^e\vert_\Omega = p$ and $\|p^e\|_{H^1(S)} \lesssim \|p\|_{H^1(\Omega)}$, \cite{stein70}.
\begin{cor}\label{cor:cont_error}
There holds
\[
\|\bu - \bu_h\|_{[H^1(\Omega)]^d} + \|h^{\frac12}(\blambda - \blambda_h)\|_{\Gamma} \lesssim h (\|\bu\|_{[H^2(\Omega)]^d} + \|\blambda\|_{[H^{\frac12}(\Gamma)]^d})
\]
\[
\|p - p_h^e\|_{\Omega} \lesssim h (\|\bu\|_{[H^2(\Omega)]^d} + \|\blambda\|_{[H^{\frac12}(\Gamma)]^d}+ \|p\|_{H^1(\Omega)})
\]
\end{cor}
\begin{proof}
The first two terms on the left hand side are bounded using the triangle inequality, 
\[
\|\bu - \bu_h\|_{[H^1(\Omega)]^d} \leq \|\bu_h - \pi_h \bu^e\|_{[H^1(\Omega)]^d} + \|\bu - \pi_h \bu^e\|_{[H^1(\Omega)]^d} 
\]
and
\[
\|\blambda - \blambda_h \|_{\Gamma}\leq  \|\blambda_h - \pi_\Gamma \blambda^e\|_{\Gamma}+\|\blambda - \pi_\Gamma \blambda^e\|_{\Gamma}
\]
and we see that the first terms of the right hand sides satisfies the bound thanks to Theorem \ref{thm:disc_error}. The second is bounded by the interpolation estimate \eqref{eq:approx} and \eqref{eq:mulitplier_approx}. For the bound on the pressure we see that by the definition of $\pi_{\mathcal{T}}$ there holds
\[
\|p - \pi_{\mathcal{T}} p\|_{\Omega_I} \lesssim h \|p\|_{H^1(\Omega_I)}
\]
and the interior bound
\[
\|p -  p^e_h\|_{\Omega_I} \lesssim \|p - \pi_{\mathcal{T}} p\|_{\Omega_I} + \underbrace{\|p^e_h - \pi_{\mathcal{T}} p\|_{\Omega_I}}_{\lesssim \|p_h - \pi_{\mathcal{T}} p\|_{0,\mathcal{T}_h}} \lesssim h (\|p\|_{H^1(\Omega_I)}+\|\bu\|_{[H^2(\Omega)]^d} + \|\blambda\|_{[H^{\frac12}(\Gamma)]^d})
\]
follows by using a triangle inequality and the discrete error bound of Theorem \ref{thm:disc_error}.
By the definition of $p_h^e$ we see that
\[
\|p - p_h^e\|_{\Omega}^2 = \|p -  p_h^e\|_{\Omega_I}^2 + \sum_{T \in \mathcal{T}_\Gamma} \|p^e - p_h^e\|^2_T.
\]
For each term of the sum in the right hand side we add and subtract $(\pi_{\mathcal{T}} p)^e$,
\begin{equation}\label{eq:pressure_split}
\|p^e - p_h^e\|_T \leq \|E_{S_h(T)} (\pi_{\mathcal{T}} p - p_h) - \bar p_h^{\Omega_{\mathcal{T}}}\|_T + \|p^e - E_{S_h(T)} \pi_{\mathcal{T}} p\|_T.
\end{equation}
For the first term in the right hand side we see that
\[
\|E_{S_h(T)} (\pi_{\mathcal{T}} p - p_h)- \bar p_h^{\Omega_{\mathcal{T}}}\|_T \lesssim \|\pi_{\mathcal{T}} p - p_h- \bar p_h^{\Omega_{\mathcal{T}}}\|_{S_h(T)}
\]
and summing over the elements we obtain
\[
\sum_{T \in \mathcal{T}_\Gamma} \|E_{S_h(T)} (\pi_{\mathcal{T}} p - p_h)- \bar p_h^{\Omega_{\mathcal{T}}}\|_T^2 \lesssim  \|\pi_{\mathcal{T}} p - p_h^e\|_{S_h(T)}^2 \lesssim \underbrace{\|\pi_{\mathcal{T}} p - p^e_h \|_{\Omega_I}}_{\lesssim \| \pi_{\mathcal{T}} p - p_h\|_{0,\mathcal{T}_h}}.
\]
The right hand side is once again bounded by Theorem \ref{thm:disc_error}.
For the second term on the right hand side of \eqref{eq:pressure_split} we add and subtract $\pi_B (p^e)$ to obtain
\[
\|p^e - E_{S_h(T)} \pi_{\mathcal{T}} p\|^2_T \lesssim \|p^e - \pi_B (p^e)\|_{B_{T,S_h}}^2 + \|E_{S_h(T)} \pi_{\mathcal{T}} p - \pi_B (p^e)\|_T^2.
\]
For the second term of the right hand side we have using the definition of $E_{S_h(T)}$ and the local stability of $\pi_{\mathcal{T}}$ for elements in $\mathcal{T}_I$,
\[
\|E_{S_h(T)} \pi_{\mathcal{T}} p - \pi_B p\|_T^2  = \|E_{S_h(T)} (\pi_{\mathcal{T}} p - \pi_B p)\|_T^2 \lesssim \|\pi_{\mathcal{T}} (p - \pi_B (p^e))\|_{S_h(T)}^2 \lesssim \|p^e - \pi_B (p^e)\|_{B_{T,S_h}}^2.
\]
It follows that 
\[
\|p^e - E_{S_h(T)} \pi_{\mathcal{T}} p\|^2_T  \lesssim \|p - \pi_B p\|_{B_{T,S_h}}^2 \lesssim h^2 |p^e|_{H^1(B_{T,S_h})}^2
\]
and the conclusion follows using the stability of the extension and the finite overlap of the balls $B_{T,S_h}$.
\end{proof}
\begin{rem}
The solutions of the Nitsche formulation \eqref{eq:Nit} satisfy similar bounds, the proof only differs in the treatment of the Nitsche terms that is a standard argument and therefore omitted.
\end{rem}
\begin{rem}
We observe that the bulk pressure does not appear in the right hand side of the bound of Theorem \ref{thm:disc_error}. On the other hand the multiplier includes the pressure forces on the boundary and therefore we expect the method to be robust with respect to the bulk pressure, but not the boundary pressure. This is illustrated in the numerical section.
\end{rem}
\begin{rem}
Observe that the estimate on the stresses is in an h-weighted $L^2$-norm and suboptimal by $O(h^{\frac12})$ in the $L^2(\Gamma)$ norm.
To obtain an estimate of order $O(h)$, but in a weaker norm, we define
\[
\|\bv\|_{-\frac12,div} := \sup_{\substack{\bw \in H_0^{div}\\ \|\bw\|_{H^1}=1}} (\bv,\bw)_\Gamma.
\]
We can then prove the estimate
\[
\|\blambda - \blambda_h\|_{-\frac12,div} \lesssim h (\|\bu\|_{[H^2(\Omega)]^d} + \|\blambda\|_{[H^{\frac12}(\Gamma)]^d}).
\]
This follows from
\[
\|\blambda - \blambda_h\|_{-\frac12,div} = \sup_{\substack{\bw \in H_0^{div}\\ \|\bw\|_{H^1}=1}} (\blambda - \blambda_h,\bw)_\Gamma
\]
and using the formulation we have 
\[
(\blambda - \blambda_h,\bw)_\Gamma = (\blambda - \blambda_h,\bw - \pi_h \bw)_\Gamma - a(\bu - \bu_h, \pi_h \bw),
\]
recalling that $b_h(\cdot,  \pi_h \bw) = 0$, since $\nabla \cdot \pi_h \bw = 0$.

For the first term we use Cauchy-Schwarz inequality and local trace inequalities, followed by \eqref{eq:approx}
\[
(\blambda - \blambda_h,\bw - \pi_h \bw^e)_\Gamma  \lesssim \|h^{\frac12}(\blambda - \blambda_h)\|_{\Gamma} (h^{-1} \|\bw - \pi_h \bw)\|_{\Omega_\Gamma} + \|\nabla (\bw - \pi_h \bw) \|_{\Omega_\Gamma}) \lesssim \|h^{\frac12}(\blambda - \blambda_h)\|_{\Gamma} .
\]
On the other hand using the Cauchy-Schwarz inequality we have for the second term
\[
a(\bu - \bu_h, \pi_h \bw) \lesssim \|\nabla (\bu - \bu_h)\|_\Omega \|\pi_h \bw\|_{H^1(\Omega)} \lesssim \|\nabla (\bu - \bu_h)\|_\Omega. 
\]
We may conclude by adding the above bounds and applying Corollary \ref{cor:cont_error}.
\[
\|\blambda - \blambda_h\|_{-\frac12,div} \lesssim \|h^{\frac12}(\blambda - \blambda_h)\|_{\Gamma} + \|\nabla (\bu - \bu_h)\|_\Omega \lesssim h (\|\bu\|_{[H^2(\Omega)]^d} + \|\blambda\|_{[H^{\frac12}(\Gamma)]^d}).
\]
\end{rem}

\section{Numerical Example}
In this section we will validate the theoretical results for the method introduced in section \ref{sec:lagange} numerically. The method has been implemented in a matlab code that can be obtained upon request from the second author. We study the convergence of the approximate solution on an academic model problem. The divergence free property is verified and the properties of the pressure extension validated. Then we consider a test case for pressure robustness and show that when weak imposition of boundary conditions are used pressure robustness fails both in the fitted and the unfitted case.
\subsection{Convergence}

We consider the disk with center at the origin and radius $r=1/2$. On this disc we consider a boundary driven solution (with $\ff ={\mathbf{0}}$):
\begin{align}
u_x= &{} 20 x y^3 \\
u_y= &{} 5x^4-5y^4\\
p=&{}60 x^2 y -20 y^3 
\end{align}
We note that as $\bu = \bu_\Gamma\neq \bf 0$ on the boundary, an additional forcing term 
\begin{equation}
l_\Gamma(\bmu) :=  \int_{\Gamma} \bu_\Gamma\cdot  \bmu ~\mbox{d}s
\end{equation}
must be added to the right-hand side of (\ref{eq:FEM_brinkman}). We also remark that since the divergence equation is integrated over the whole of the cut elements, whereas the elliptic term is only integrated over the cut parts, severe ill-conditioning may occur. To remedy this, we add 
an additional stabilizing term on the whole of the cut elements as discussed in Remark \ref{rem:ghost},
\begin{equation}
 a_\text{stab}(\bu_h,\bv_h) := \sum_{T\in\mathcal{T}_\Gamma}h^2_T(\nabla\times \bu_h,\nabla \times \bv_h)_{T}.
\end{equation}
Here we chose a curl based operator to add a minimal perturbation to the method. This term improved the conditioning sufficiently to solve the linear system, but it would not be adequate for use with Nitsche's method.
We set the stabilization parameter for the Lagrange multiplier $\gamma =1$.

In Fig. \ref{fig:velocity} we show the velocity errors in $L_2(\Omega)$ and $H^1(\Omega)$. The dashed line has inclination 1:1 and the dotted line 2:1. We note that the obtained $L_2$ convergence is slightly better than expected but seems to tend to the expected 2:1 on finer meshes. The $H^1$ convergence is $O(h)$ as expected.
Fig.  \ref{fig:pressure} shows the convergence of the pressure in $L_2(\Omega_I)$ (the cut elements being omitted) and the multiplier $\blambda_h$ in $L_2(\Gamma)$. We observe $O(h)$ convergence in both cases. For computing the value of $\blambda$ we used the discrete normals. In the examples we have used $h:=1/\sqrt{N}$ where $N$ is the number of nodes on the macro elements of $ \Omega_\mathcal{T}$.

In Fig. \ref{fig:presselev} we show elevations of the computed pressure and exact (represented as piecewise constant) pressure on $\Omega$. Note that the discrete pressure is zero on the cut elements. In Fig. \ref{fig:lambdaelev} we show an elevation of the discrete multiplier and the interpolated exact multiplier represented on $\Omega_\Gamma$, and
in Fig. \ref{fig:divergence} we show the divergence the approximate velocities on $\Omega$ which is indeed small everywhere.

To show the effect of recovering the pressure on the cut elements, we consider the simplest possible method: a constant extension from the closest uncut neighbour (closeness based on distance between centroids).  In Fig. \ref{fig:extension} we show the pressure before and after recovery, and, finally in Fig. \ref{fig:pressure2} we show convergence of the recovered pressure in $L_2(\Omega)$, which is $O(h)$ , as expected.

\subsection{Pressure robustness}

We consider a problem with Coriolis force, following John et al. \cite{JLMNR17}. To this end, we use the model 
\begin{equation}
-\Delta\bu +\nabla p + 2{\boldsymbol \omega}\times \bu = \boldsymbol f
\end{equation}
in two dimension with ${\boldsymbol \omega} = (0,0,\omega)$. Then the magnitude of $\omega$ will only affect the pressure, cf. \cite{JLMNR17}. If the discrete scheme possesses this quality, we may say that it is {\em pressure robust}. The element underlying our method has this quality, but when applying weak boundary conditions, we find that the control of velocities on the boundary is insufficient to retain it (cut {\em and} standard formulation). An increase in $\omega$ will eventually lead to disturbance of the boundary velocity which spreads to the interior. To illustrate this, we consider the same domain as in the previous section and apply boundary conditions $\bu = (1,0)$ and ${\boldsymbol f}=\bf 0$. In Figs. \ref{fig:1000} and \ref{fig:10000} we show elevations of the $y-$component of the computed flow. The instability increases from being virtually non-noticeable at $\omega=0$ to giving completely wrong velocity solution at $\omega=10000$ ($y-$component non zero throughout $\Omega$).
Note also that the disturbance is induced by errors at the boundary which seem to increase linearly with $\omega$.

\section{Conclusions}
We developed two different cut finite element methods for the approximation of incompressible viscous flow using pointwise divergence free elements. The key observations was to enforce the divergence free condition globally on the computational mesh and separate the bulk pressure approximation from the boundary pressure approximation. Optimal error bounds were derived with upper bounds for the velocity error independent of the pressure regularity. The addition of degrees of freedom for the boundary pressure leads to a slightly more complicated method, on the other hand fewer integrals over cut elements have to be evaluated so the assembly of the system may not necessarily be more expensive. In view of the rather complete analysis obtained we consider this a small price to pay. One possible extension of the present approach is to $H(div)$-conforming approximation using the Raviart-Thomas space. This element has similar stability structure as the $H^1$-conforming element considered herein and  one would therefore expect to obtain robustness and optimality for similar unfitted approximations of Darcy's equation.

\section*{Acknowledgements.}
The authors wish to thank the anonymous reviewers whose constructive comments helped make this a better paper.
This research was supported in part by the Swedish Research
Council Grants Nos.\  2017-03911, 2018-05262,  2021-04925,  2022-03908,  and the Swedish
Research Programme Essence.  EB was supported in part by the EPSRC grants EP/P01576X/1 and EP/T033126/1.

\section*{Appendix}

\subsection{Construction of the Finite Element Space}  \label{sec:construction-element}
We here present a construction of the space $\bV_h$ introduced in Section \ref{sec:FEMconst}.  The 
finite element space $\bV_h(T) = \bV_h|_T$ consists of the linear vector valued polynomials together with face bubbles, one 
for each face of the element,  which have constant divergence. The face bubbles are continuous piecewise linear functions on a certain partition of $T$ into simplexes. To construct the face bubbles we first construct a bubble function $\varphi_T$ on $T$. To that end we let $\mcS(T)$ be a partition of $T$ into sub-simplexes 
constructed by inserting a node $\bx_T$ in the interior of $T$ and then adding edges between the vertices of $T$ and $\bx_T$. Typically, we chose $\bx_T$ as the barycenter of $T$. This choice is not necessary,  but we will see that it is leads to 
very simple expressions for the basis functions.  We 
let the bubble function $\varphi_T$ be continuous piecewise linear on $\mcS(T)$, equal 
to one in $\bx_T$, and zero in the vertices of $T$. Next consider a face $F$ of $T$ and let $T_F \in \mcS(T)$ be the sub-simplex associated with $F$. We begin by inserting a node in an arbitrary point $\bx_F$ in the interior of the 
face  $F$ and then we partition $F$ into sub-simplexes by inserting edges from $\bx_F$ to the vertices of $F$. We 
let $\mcS(T_F)$ be the partition of $T_F$ into sub-simplexes obtained by inserting edges from $\bx_T$ to the vertices of $F$ and $\bx_F$. On $\mcS(T_F)$ we let $\varphi_F$ be the continuous piecewise linear face bubble,  which is one in $\bx_F$ and zero in $\bx_T$ and the vertices of $F$. We extend $\varphi_F$ to $T$ by zero. See Figure \ref{fig:element} for the partitions of the element $T$ into sub-simplexes. The vector valued face bubble $\bvarphi_F$ is now defined by
\begin{equation*}
\bvarphi_F := \bnu_F \varphi_F + \alpha \bvarphi_T \mbox{ with } \bvarphi_T := (\bx_T - \bx_{T,F} ) \varphi_T, 
\end{equation*}
here $\bnu_F = (\bx_F - \bx_T)/\| \bx_F - \bx_T \|_{\IR^d}$, $\bx_{T,F}$ is the vertex of $T$ opposite to 
$F$, and $\alpha\in \IR$ is a parameter that we will see can be determined such that the divergence 
\begin{equation*}
\nabla \cdot \bvarphi_F =  \nabla  \cdot (\bnu_F \varphi_F+ \alpha ( (\bx_T - \bx_{T,F} ) \varphi_T )
\end{equation*}
is constant on $T$. 

To verify that we can indeed find such an $\alpha$ we make the following observation. Consider a general simplex 
$\tT$. Let $\tF$ be a face of $\tT$ and $\tbx$ the vertex opposite to $\tF$. Let $\tvarphi$ be the linear 
function on $\tT$, which is one in $\tbx$ and zero on $\tF$.  Let $\tby$ be one of the nodes of $\tT$ that belong 
to $\tF$ and consider the vector valued function $\tbvarphi = (\tbx - \tby) \tvarphi$. Then we have
\begin{equation*}
\nabla \cdot \tbvarphi   = (\tbx - \tby) \cdot \nabla \tvarphi = (\tbx - \tby) \cdot \frac{\tilde{\bn} }{\tilde{l}}, 
\end{equation*}
where $\tilde{\bn}$ is the unit normal of $\tF$ directed towards $\tbx$ and $\tilde{l}$ is the distance from $\tbx$ to 
$\tF$. We then note that
\begin{equation*}
(\tbx - \tby) \cdot \tilde{\bn} = \|\tbx - \tby\|_{\IR^d} \sin \tilde{\theta}, \qquad  \tilde{l} =  \|\tbx - \tby\|_{\IR^d} \sin \tilde{\theta},
\end{equation*}
where $\tilde{\theta}$ is the angle between the vector $\tbx - \tby$ and $\tF$, and we may conclude that 
\begin{align*}
\nabla \cdot \tbvarphi = 1.
\end{align*}
With this observation at hand we have
\begin{align*}
\nabla \cdot \bvarphi_T = 1 \quad \text{ on $ \mcS(T) \setminus T_F$} 
\end{align*}
and using the fact that $(\nabla \cdot \bvarphi_T, 1)_T = (\bn \cdot \bvarphi_T, 1)_{\partial T} = 0$, we get 
\begin{align*}
\nabla \cdot \bvarphi_T = - \frac{|T| - |T_F|}{|T_F|} \quad \text{on $T_F$}.
\end{align*}
Using again the same observation we have 
\begin{align*}
\nabla \cdot \bvarphi_F = \frac{1}{\| \bx_F - \bx_T \|_{\IR^d}} \nabla \cdot ( (\bx_F - \bx_T ) \varphi_F ) 
= \frac{1}{\| \bx_F - \bx_T \|_{\IR^d}}
\quad \text{on $\mcS(T_F)$}
\end{align*}
and by definition
\begin{align*}
\nabla \cdot \bvarphi_F = 0 \quad \text{on $T \setminus T_F$}.
\end{align*}
Since $\nabla \cdot \bvarphi_T = \nabla \cdot (\alpha (\bx_T - \bx_{T,F} ) \varphi_T ) = \alpha$ on 
$T \setminus T_F$ we seek $\alpha$ such that 
\begin{align*}
\frac{1}{\| \bx_F - \bx_T \|_{\IR^d}}  - \alpha  \frac{|T| - |T_F|}{|T_F|} = \alpha,
\end{align*}
which gives 
\begin{align*}
\alpha =  \frac{|T_F|}{|T|}   \frac{1}{\| \bx_F - \bx_T \|_{\IR^d}}  
\end{align*}
and 
\begin{align*}
\bvarphi_F = \bnu_F \varphi_F +  \frac{|T_F|}{|T|}   \frac{\bx_T - \bx_{T,F}}{\| \bx_F - \bx_T \|_{\IR^d}}  \varphi_T,
\end{align*}
with divergence 
\begin{align*}
\nabla \cdot \bvarphi_F = \alpha \quad \text{on $T$}.
\end{align*}
In the case when $\bx_T$ is the barycenter of $T$ we have the identity 
\begin{align}
|T_F| = \frac{1}{d+1}|T|
\end{align}
and we get the simplified expression
\begin{align*}
\bvarphi_F = \bnu_F \varphi_F +  \frac{1}{d+1}   \frac{\bx_T - \bx_{T,F}}{\| \bx_F - \bx_T \|_{\IR^d}}  \varphi_T.
\end{align*}
This is the explicit expression for the finite element face bubbles and coincides with the discussion of \cite[Section 3]{BCH20}.

Finally, to obtain a global continuous finite element space we chose $\bx_F$ on an interior face $F$ shared by two 
elements $T$ and $T'$ as the intersection between the line that passes through the barycenters of $T$ and $T'$ and the face $F$, see Figure \ref{fig:elementpair}. Then we have $\bnu_F = -\bnu'_F$ and therefore we can define a global continuous basis function associated with the face $F$

For faces on the boundary we take $\bx_F$ to be the barycenter of $F$.

\subsection{Proof of Estimate \eqref{eq:disc_interp}}
It is enough to consider one component $\Gamma_i$. Define $\mathcal{T}_{\Gamma_i}:= \{T \in \mathcal{T}_\Gamma: T\cap \Gamma_i \ne \emptyset\}$. By the definition of the average we have for all $\by \in \mathbb{R}^d$
\[
h^{\frac12}\|\bmu - \bar \bmu^{\Gamma_i}\|_{\Gamma_i} \leq h^{\frac12} \|\bmu - \by\|_{\Gamma_i} \lesssim \|\bmu - \by\|_{\mathcal{T}_{\Gamma_i}}.
\]
Now fix $\by = \bmu\vert_T$ for some $T \in \mathcal{T}_{\Gamma_i}$. Clearly then for all $T \in \mathcal{T}_{\Gamma_i}$ $\|\bmu - \by\|_T \lesssim \sum_{F \in \mathcal{F}_i(\mathcal{T}_{\Gamma_i})} \|h^{\frac12} [\bmu]\|_F^2$. Since $M$ is bounded we conclude by summing over $T \in \mathcal{T}_{\Gamma_i}$.
\subsection{Construction of $\bvarphi$}
It is sufficient to consider one component of $\bvarphi$. Let $\varphi$ denote any component of $\bvarphi$.
Let $B_i$ be the largest ball with center $\bx_i$ on $\Gamma_i$, such that $B_i \cap \Omega_{\mathcal{T}} \subset P_i$,  see Figure \ref{fig:patch}. We assume that its radius is $r_i$. By the construction of $P_i$ we may assume that $r_i = O(M h)$.  Define the function $\phi(\bx) = max(1 - |\bx-\bx_i|/r_i, 0)$. It follows that $\phi(\bx_i) = 1$ and $\phi(\bx) =  0$ for $\bx \ne B_i$. It follows by construction that $\sup_{\bx \in P_i}|\nabla\phi(\bx)| \lesssim (M h)^{-1}$ and there exists $c_0>0$ so that $\overline \phi^{\Gamma_i} > c_0$. Let $\psi_h = i_h \phi$, where $i_h$ is the nodal interpolant in the vertices of the mesh, and normalize  with $0 < c_1 = \overline \psi_h^{\Gamma_i}$ so that $\phi_h := c_1^{-1}  \psi_h$ and $\overline\phi_h^{\Gamma_i} = 1$. We may then take $\varphi = h a \phi_h$ for some $a \in \mathbb{R}$. It remains to prove that $\|\nabla \varphi\|_{\Gamma_i} + (M h)^{-\frac12} \|\nabla \varphi\|_{P_i} \lesssim M^{-1} \|a\|_{\Gamma_i}$.
By the construction of $\varphi$ we have since $\|\nabla \phi_h\|_{L^\infty(P_i)} \lesssim c_1^{-1} \|\nabla \phi\|_{L^\infty(P_i)}$,
\[
\|\nabla \varphi\|_{\Gamma_i} \lesssim c_1^{-1} a h \|\nabla \phi\|_{L^\infty(P_i)} (M h)^{\frac{d-1}{2}} \lesssim c_1^{-1} a h (M h)^{-1}  (M h)^{\frac{d-1}{2}}. 
\]
Since $a  (M h)^{\frac{d-1}{2}} \lesssim \|a\|_{\Gamma_i}$ we conclude that
\[
\|\nabla \varphi\|_{\Gamma_i} \lesssim M^{-1} \|a\|_{\Gamma_i}.
\]
On the patch $P_i$ we have similarly
\[
\|\nabla \varphi\|_{P_i} \lesssim c_1^{-1} a h \|\nabla \phi\|_{L^\infty(P_i)} (M h)^{\frac{d}{2}} \lesssim c_1^{-1} h^{\frac12} M^{-\frac12} \|a\|_{\Gamma_i}.
\]

\bibliographystyle{abbrv}
\bibliography{references}
\newpage
\begin{figure}
\begin{center}

\includegraphics[scale=0.4]{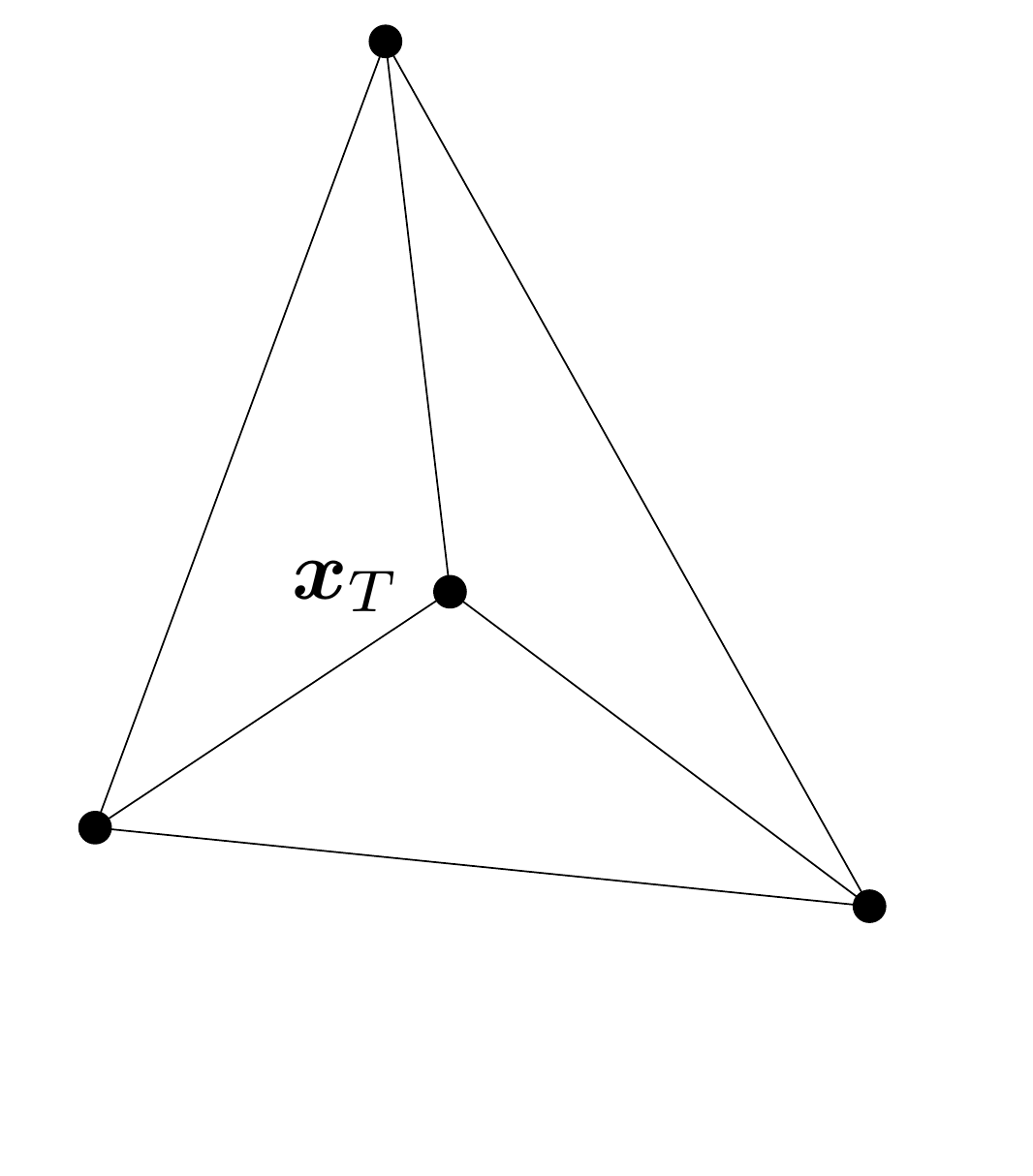}
\includegraphics[scale=0.4]{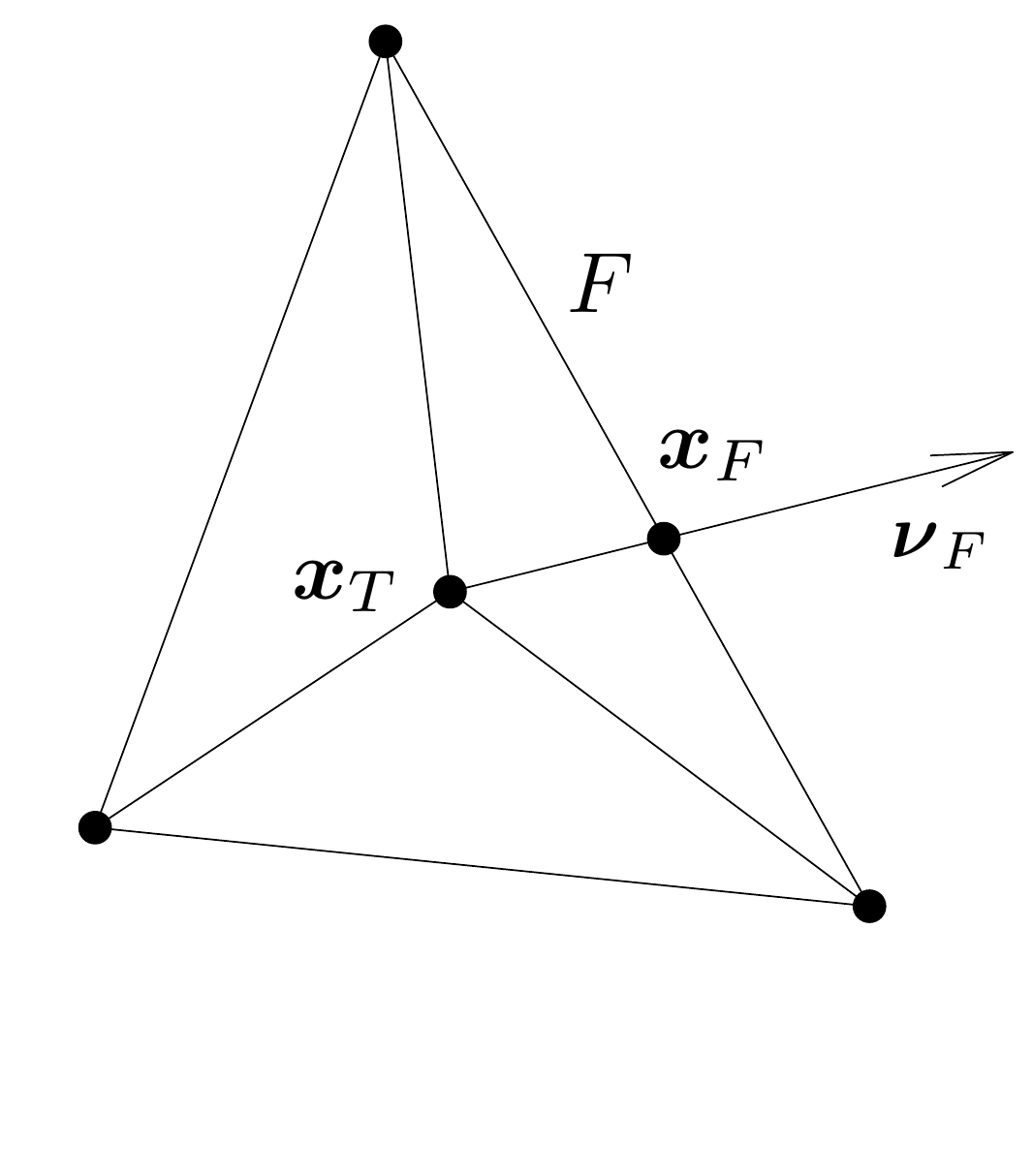}
\includegraphics[scale=0.4]{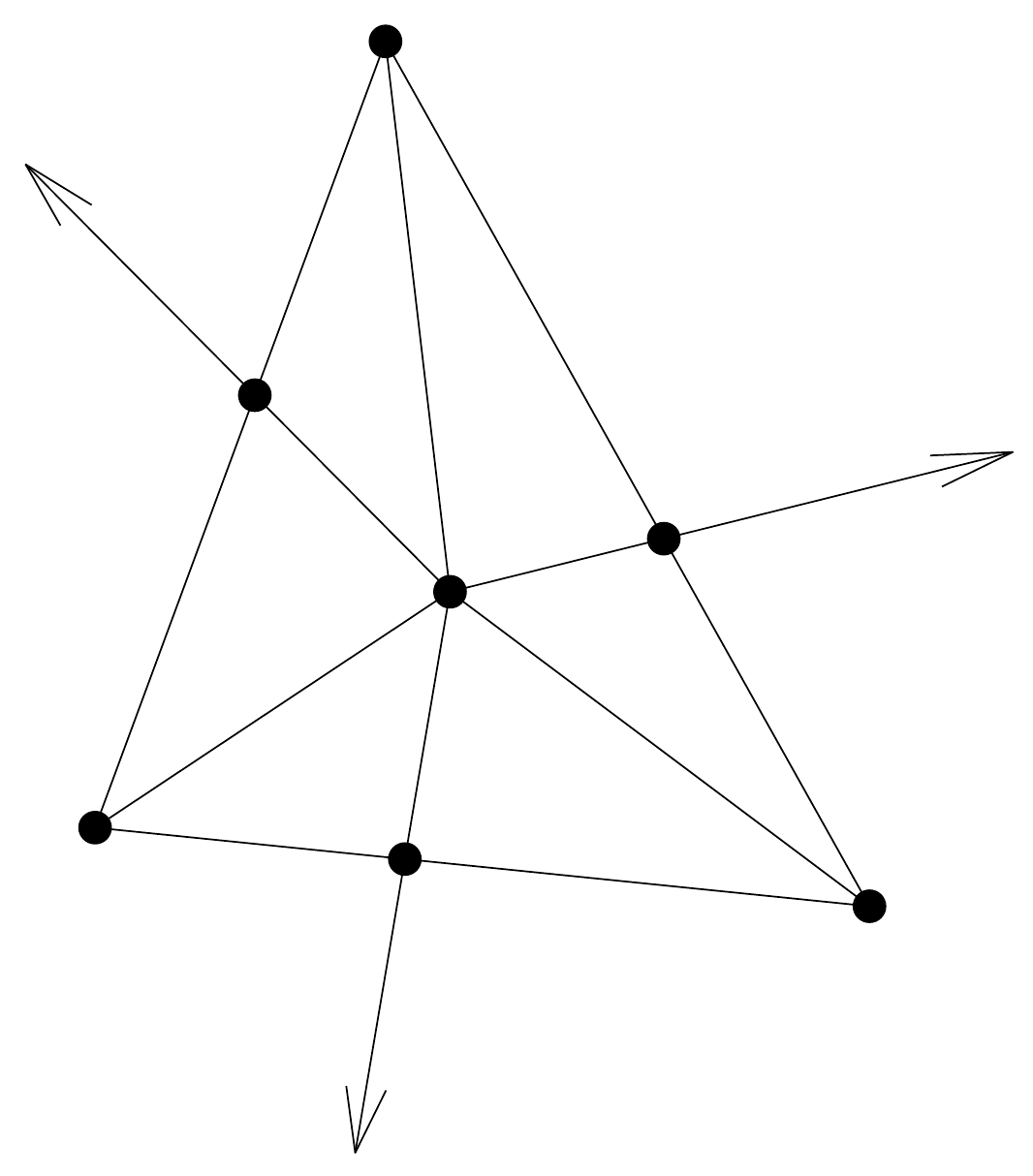}
\caption{From left to right: The partition for the element bubble $\varphi_T$, the partition for the face bubble $\bvarphi_F$, and the partition for all the degrees of freedom of the element.}
\label{fig:element}

\end{center}
\end{figure}

\begin{figure}
\begin{center}

\includegraphics[scale=0.4]{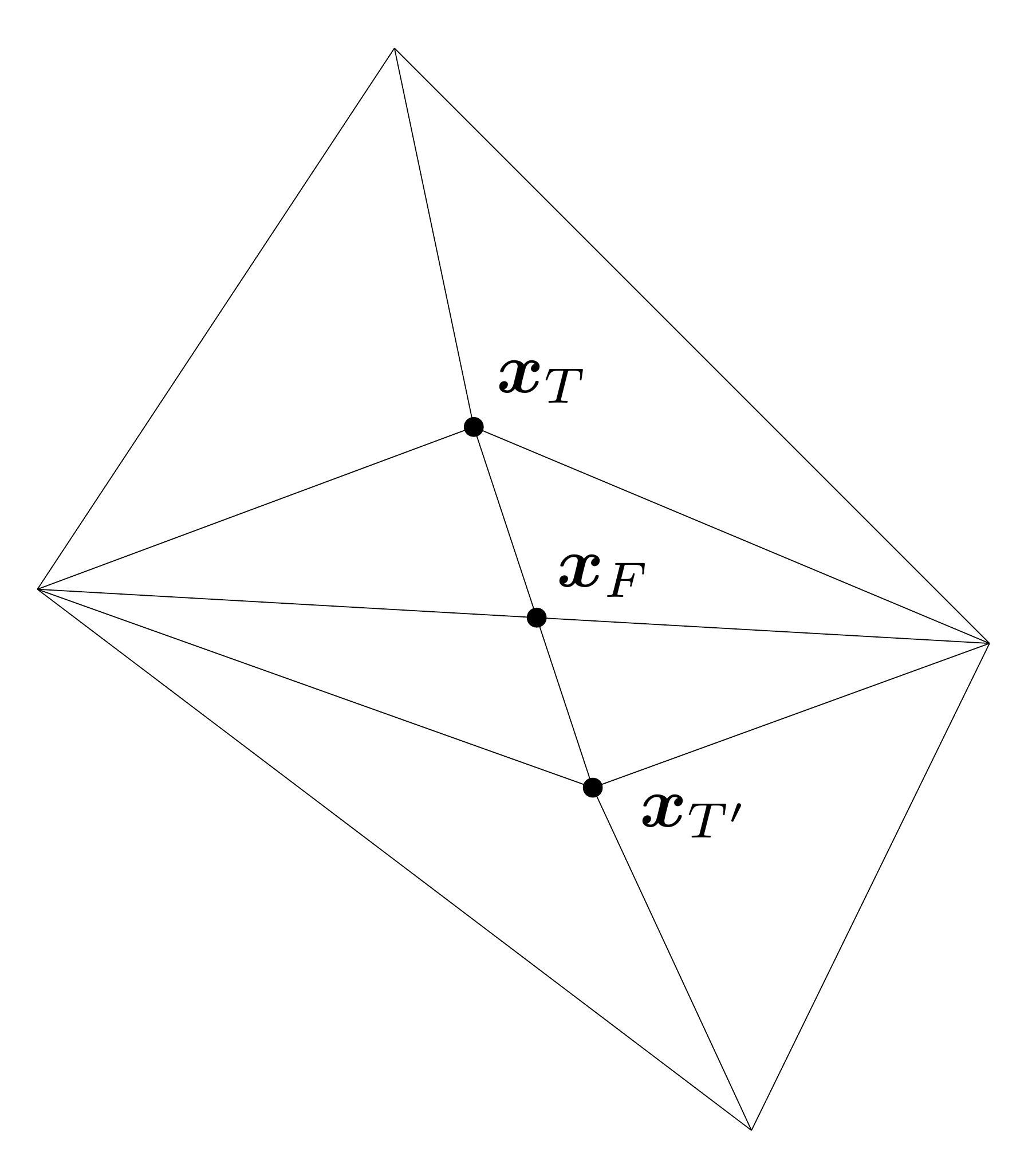}

\caption{The definition of $\bx_F$ for a face shared by elements $T$ and $T'$ as the intersection of the face $F$ and the line between the barycenters $\bx_T$ and $\bx_{T'}$, which enables us to construct a global continuous velocity space.}  
\label{fig:elementpair}

\end{center}
\end{figure}

\begin{figure}
\begin{center}

\includegraphics[scale=0.2]{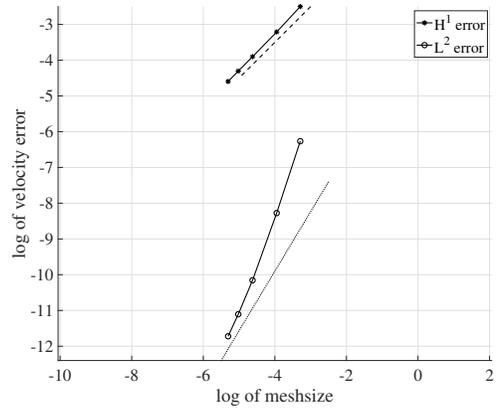}

\caption{Convergence of the velocity. \label{fig:velocity}}  

\end{center}
\end{figure}
\begin{figure}
\begin{center}

\includegraphics[scale=0.2]{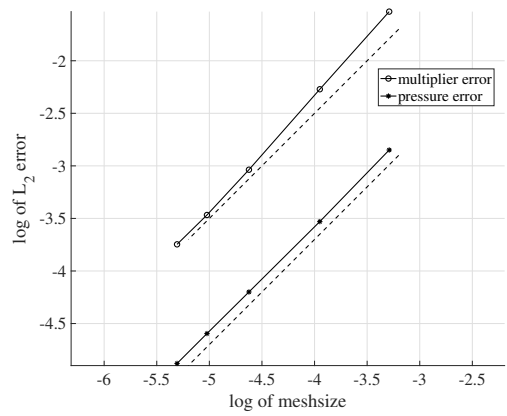}

\caption{Convergence of the pressure and boundary multiplier. \label{fig:pressure}}  

\end{center}
\end{figure}
\begin{figure} 
\begin{center}

\includegraphics[scale=0.15]{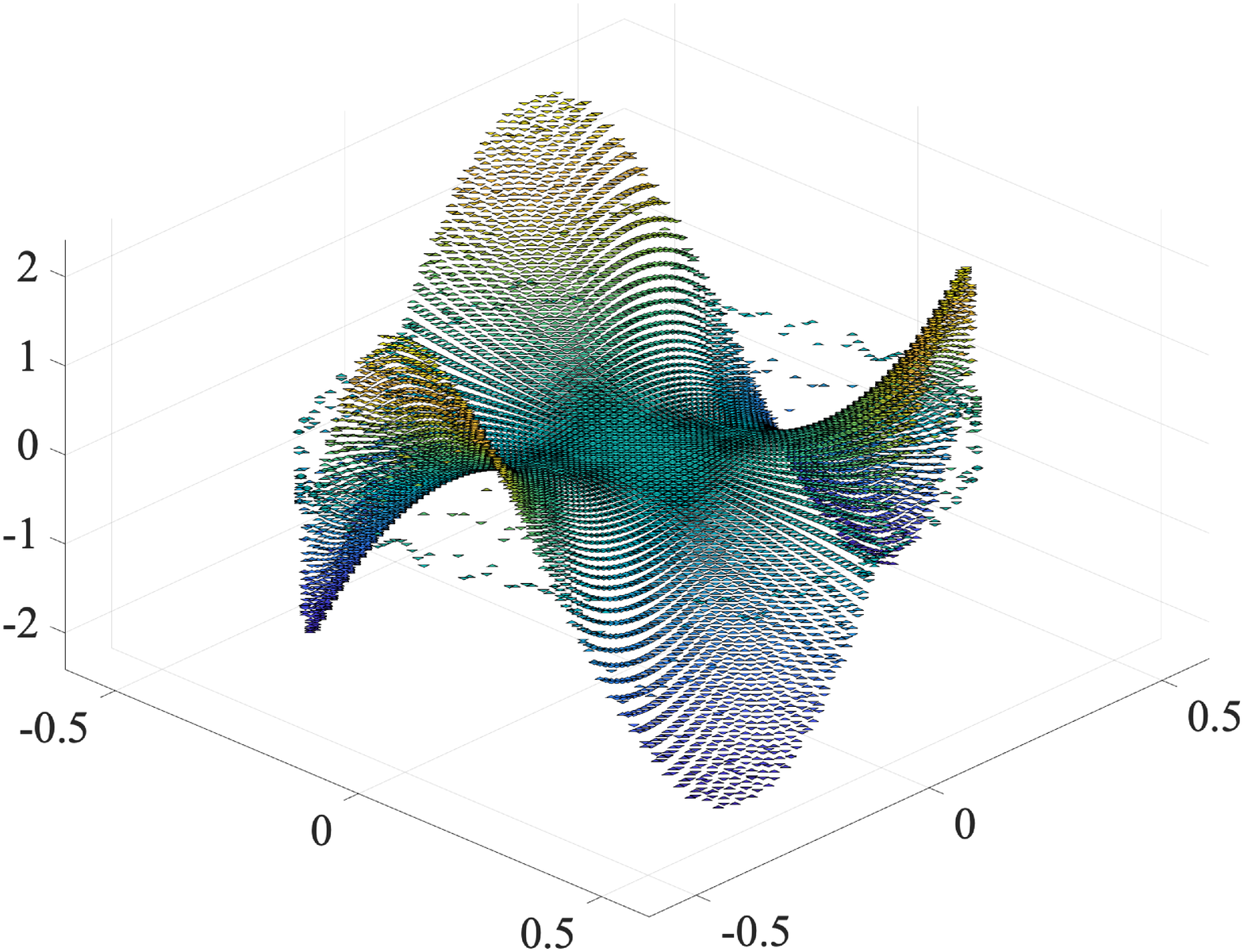}
\includegraphics[scale=0.15]{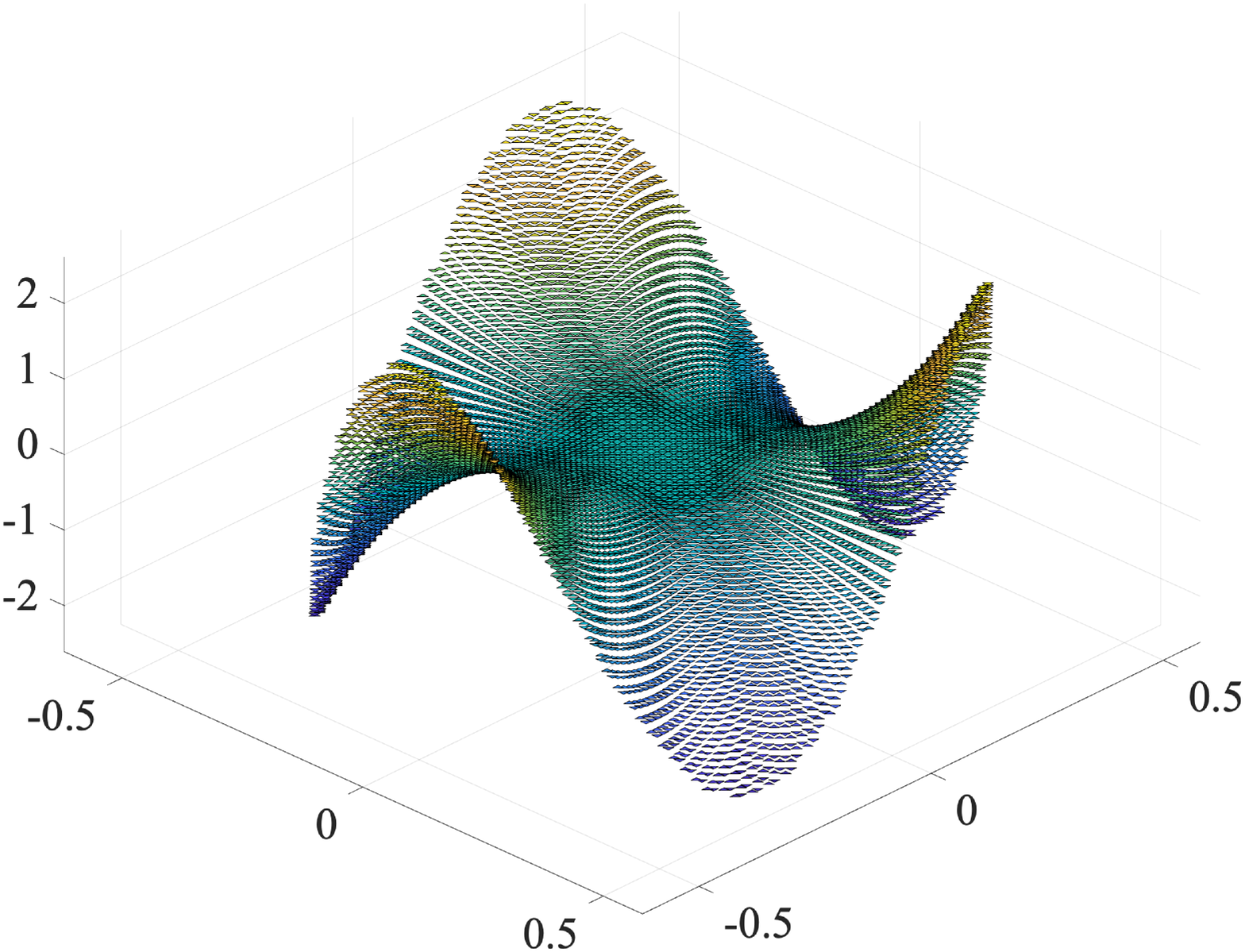}
\caption{Elevation of the discrete and interpolated exact pressure on a particular mesh.\label{fig:presselev}}  

\end{center}
\end{figure}
\begin{figure}
\begin{center}

\includegraphics[scale=0.15]{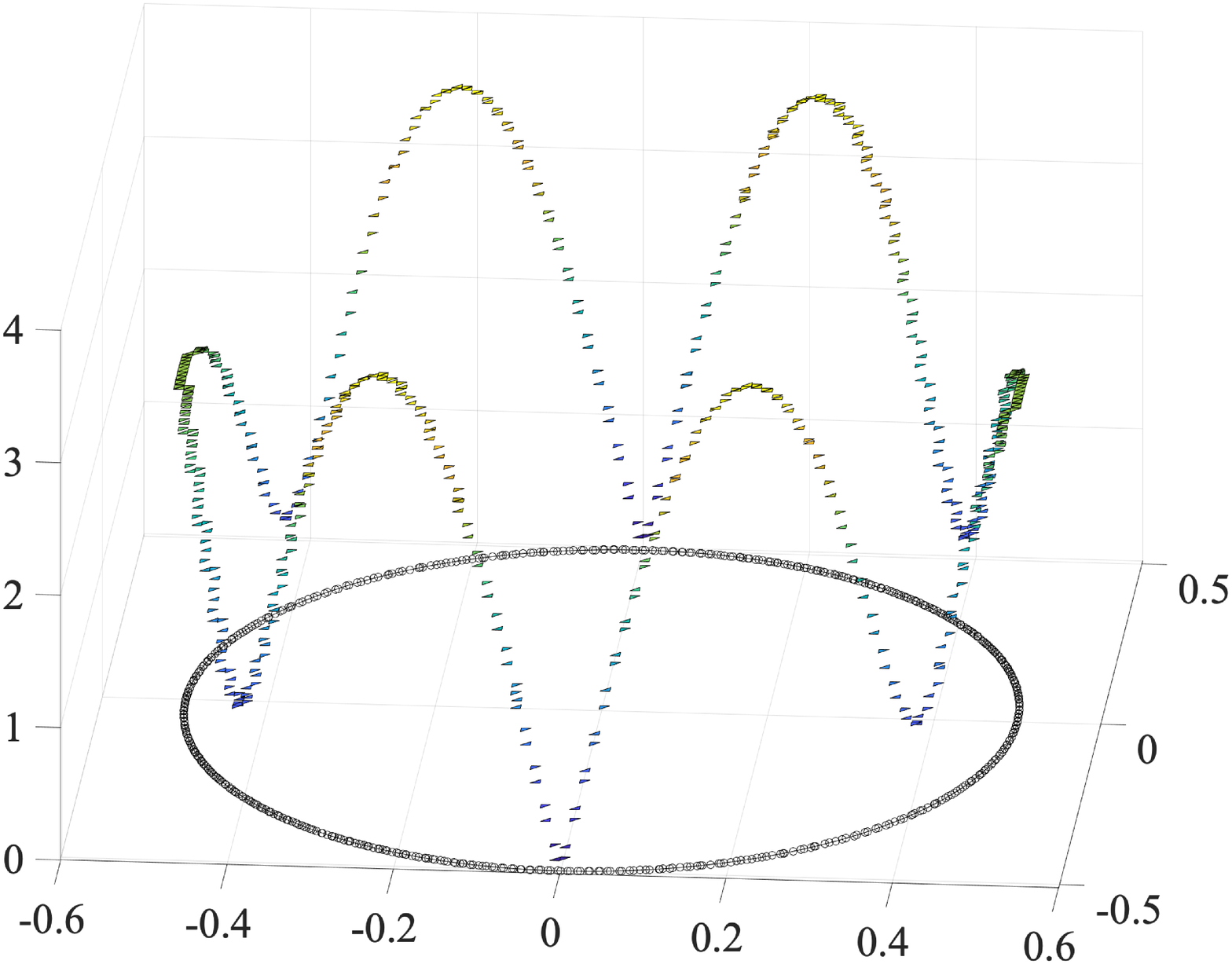}
\includegraphics[scale=0.15]{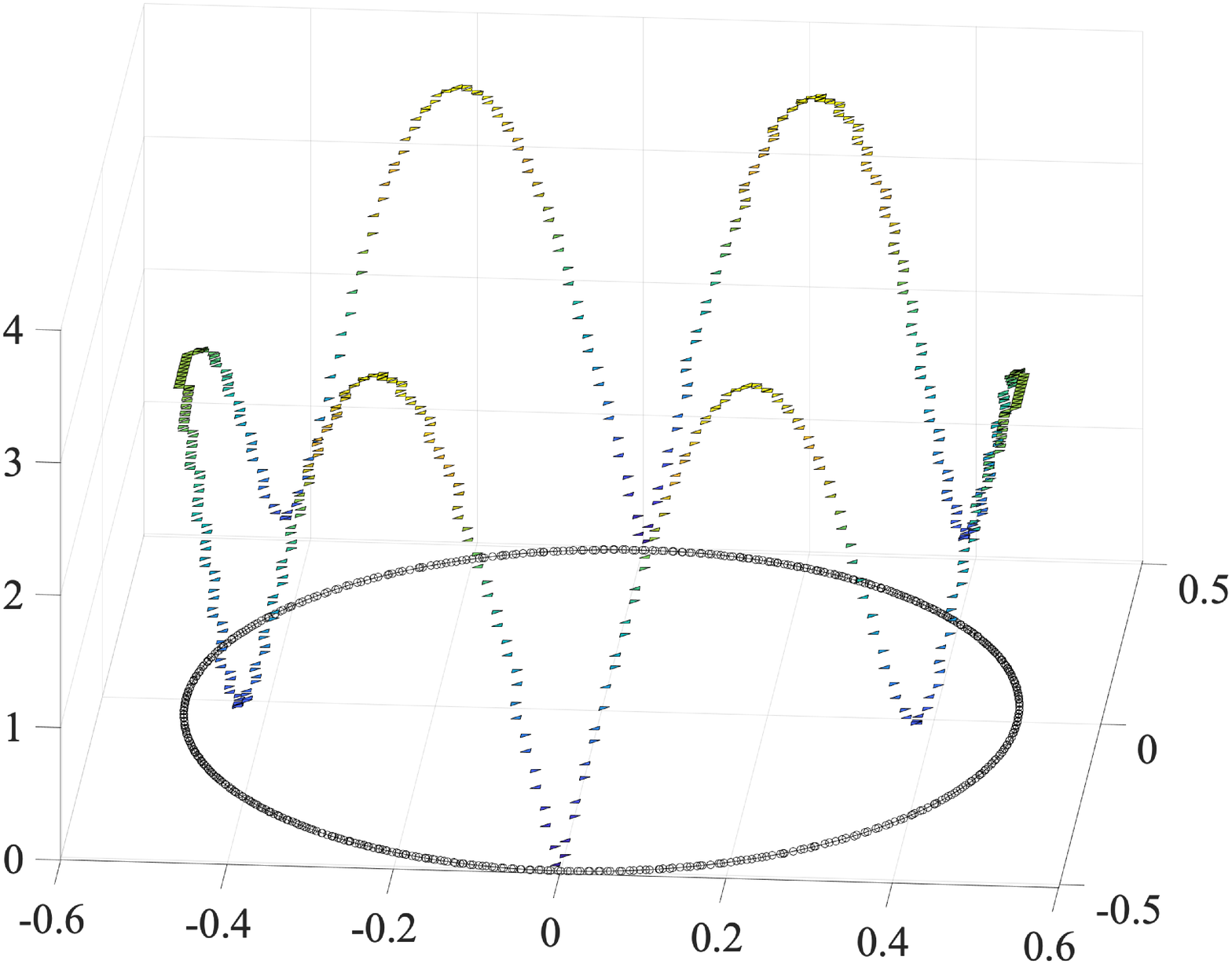}
\caption{Elevation of the discrete and interpolated exact multiplier $\blambda$ on a particular mesh. \label{fig:lambdaelev}}  

\end{center}
\end{figure}
\begin{figure}
\begin{center}

\includegraphics[scale=0.3]{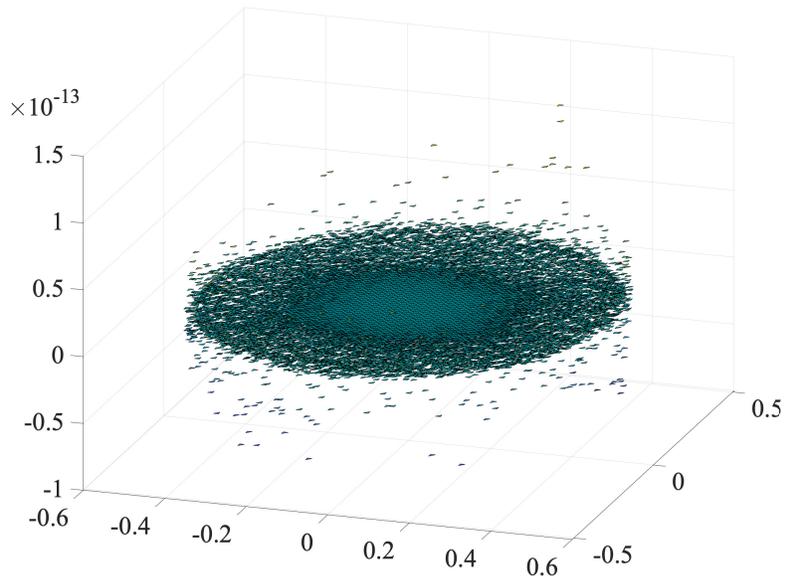}
\caption{Computed divergence on a particular mesh. \label{fig:divergence}}  

\end{center}
\end{figure}

\begin{figure}
\begin{center}

\includegraphics[scale=0.2]{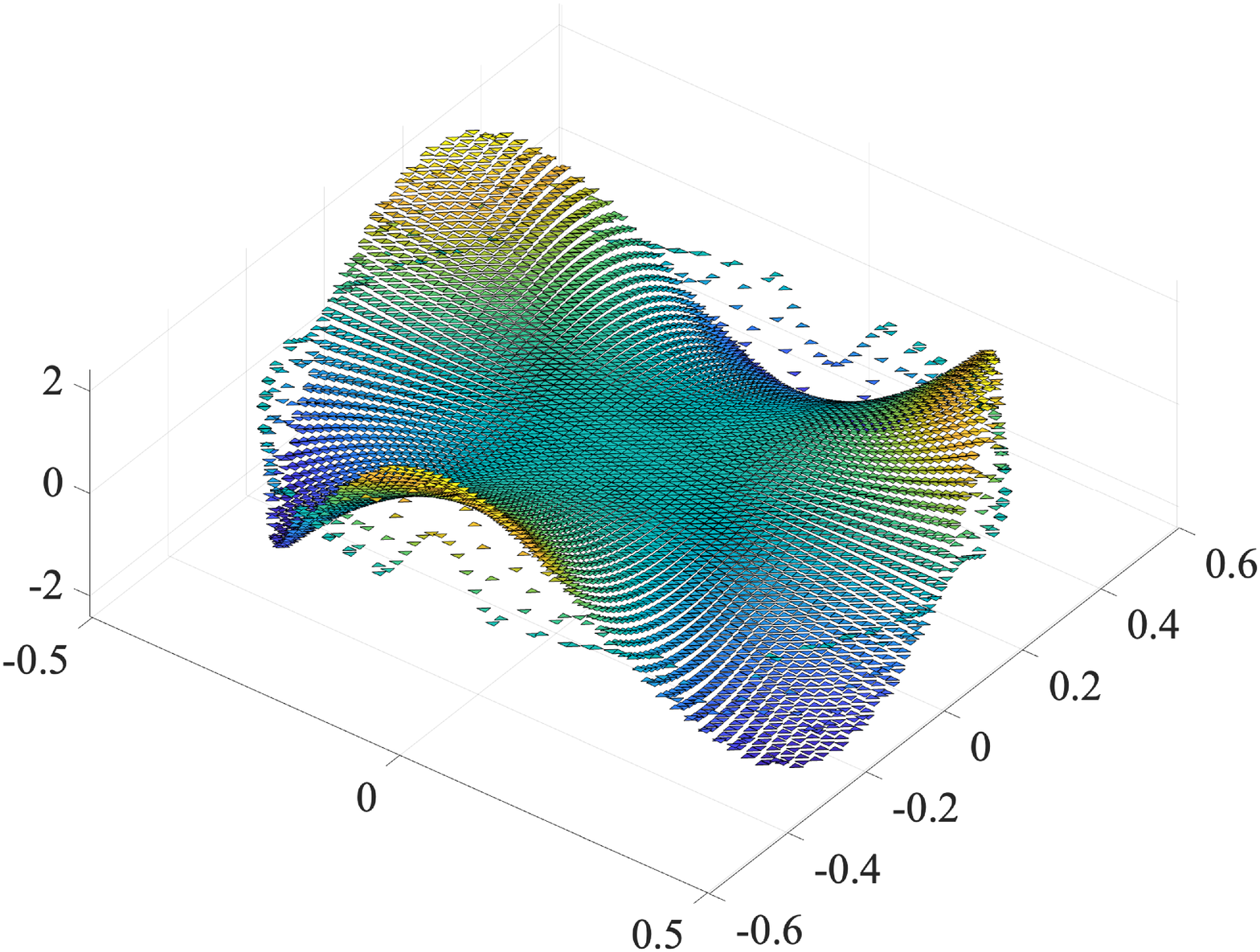}\includegraphics[scale=0.2]{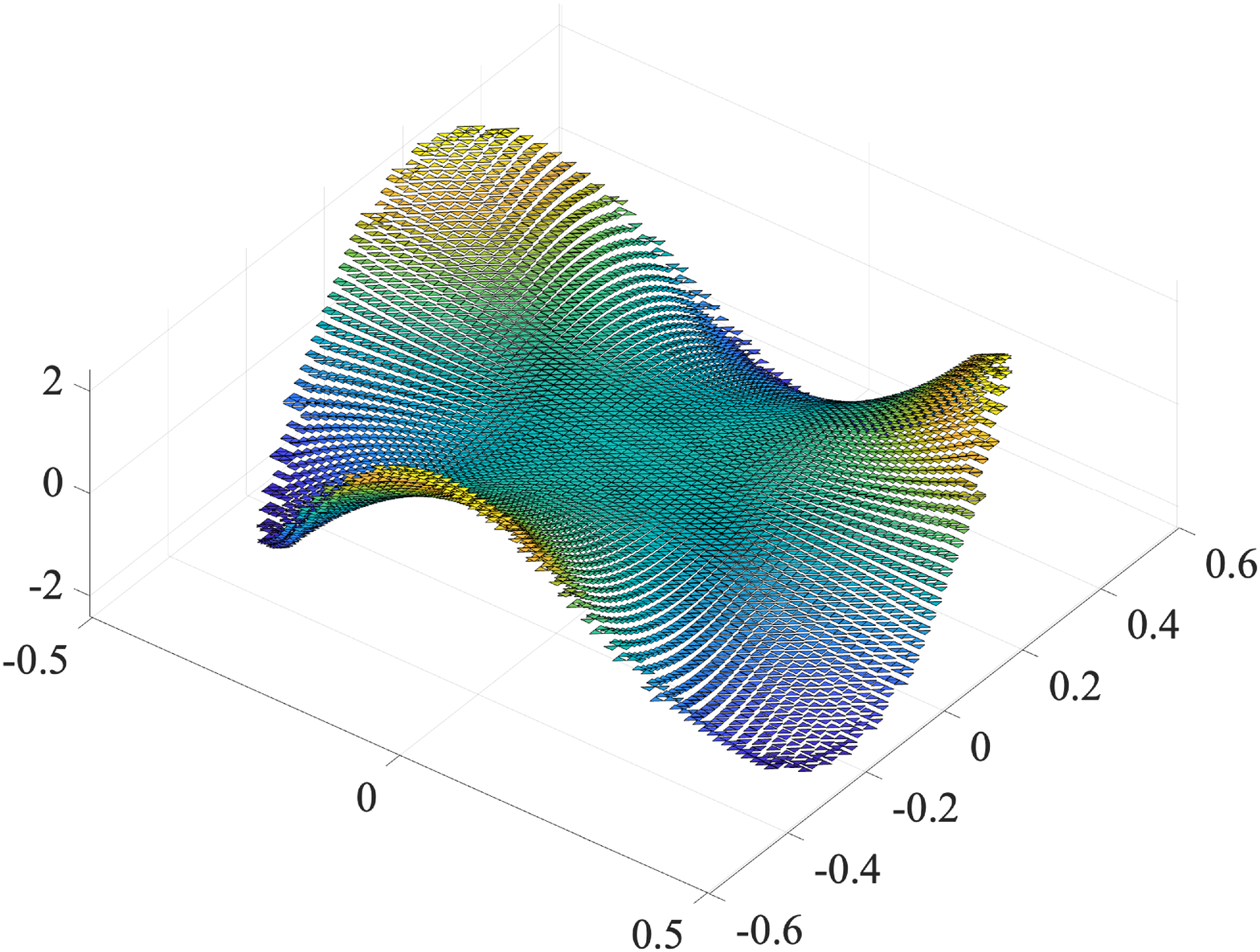}

\caption{Pressure and extended pressure on a particular mesh. \label{fig:extension}}  

\end{center}
\end{figure}

\begin{figure}
\begin{center}

\includegraphics[scale=0.3]{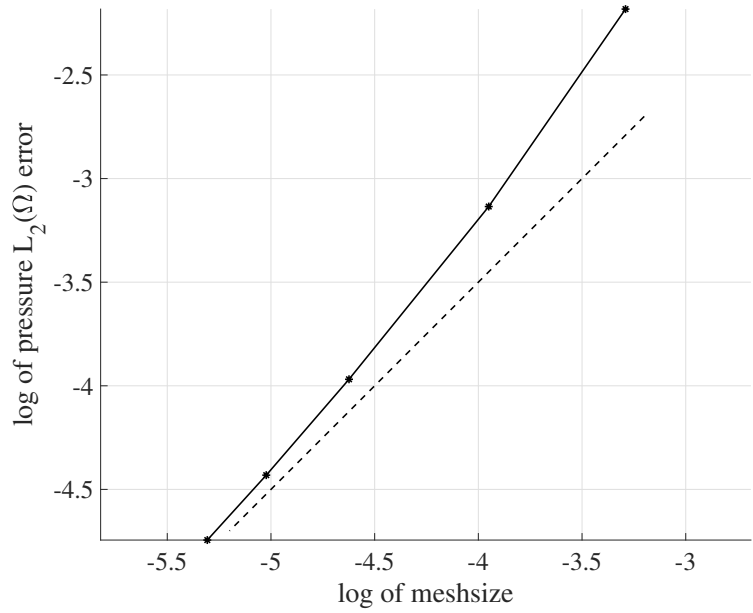}
\caption{Convergence of the extended pressure on $\Omega$. \label{fig:pressure2}}  

\end{center}
\end{figure}

\begin{figure}
\begin{center}

\includegraphics[scale=0.23]{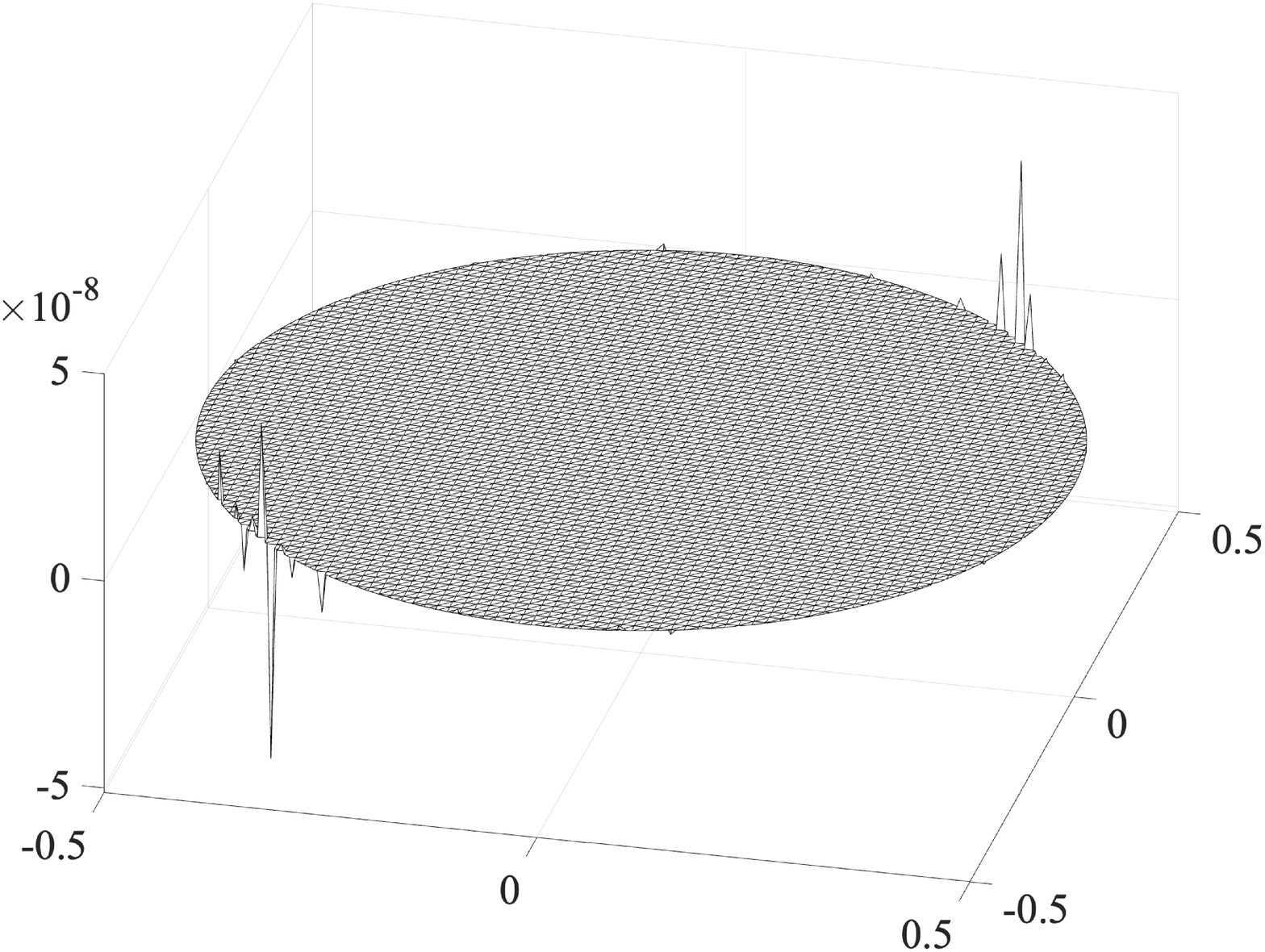}\includegraphics[scale=0.23]{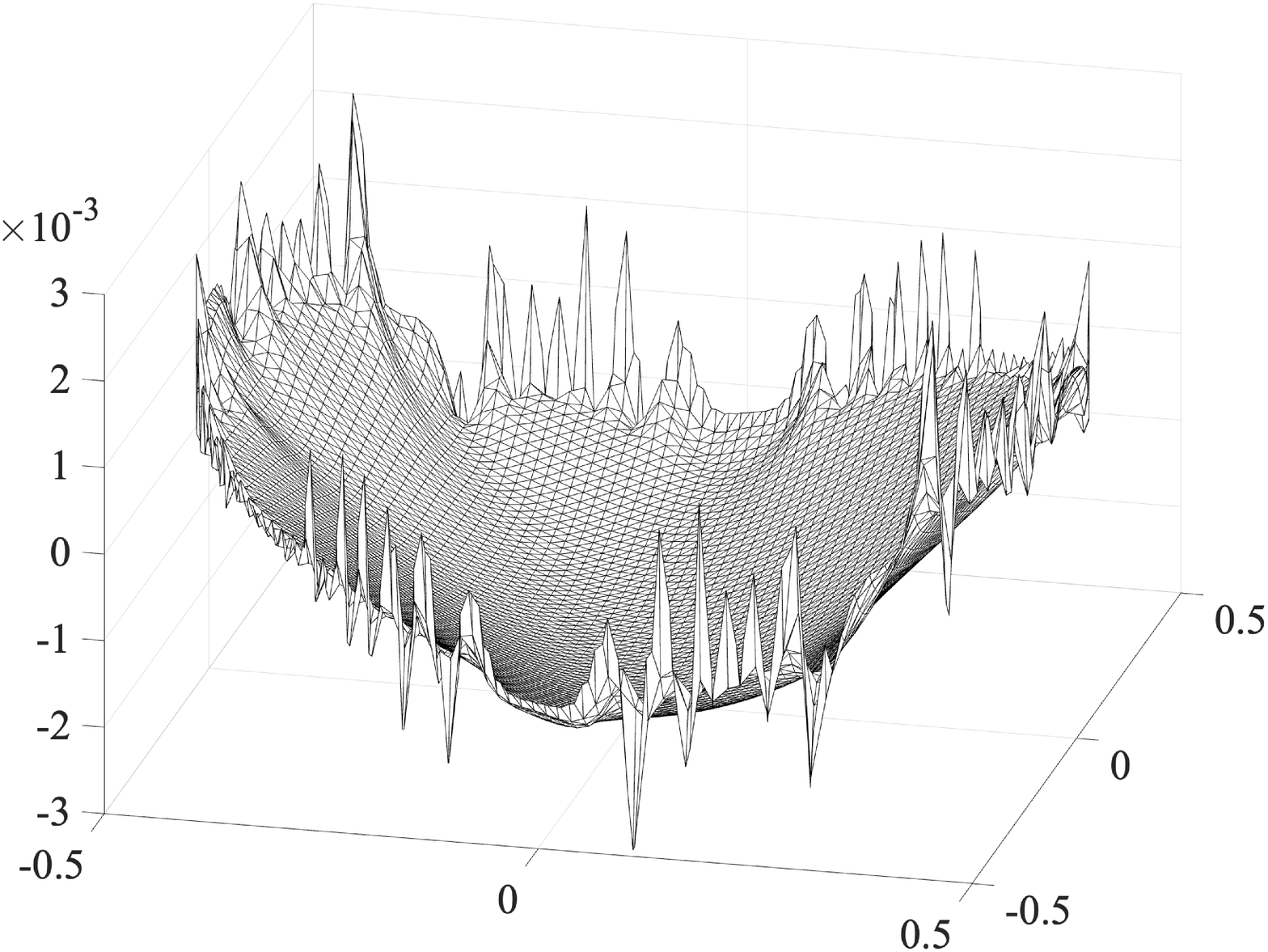}

\caption{Velocity $u_y$ on a particular mesh. $\omega = 0$ (left) and $\omega = 100$ (right).  \label{fig:1000}}  

\end{center}
\end{figure}
\begin{figure}
\begin{center}

\includegraphics[scale=0.23]{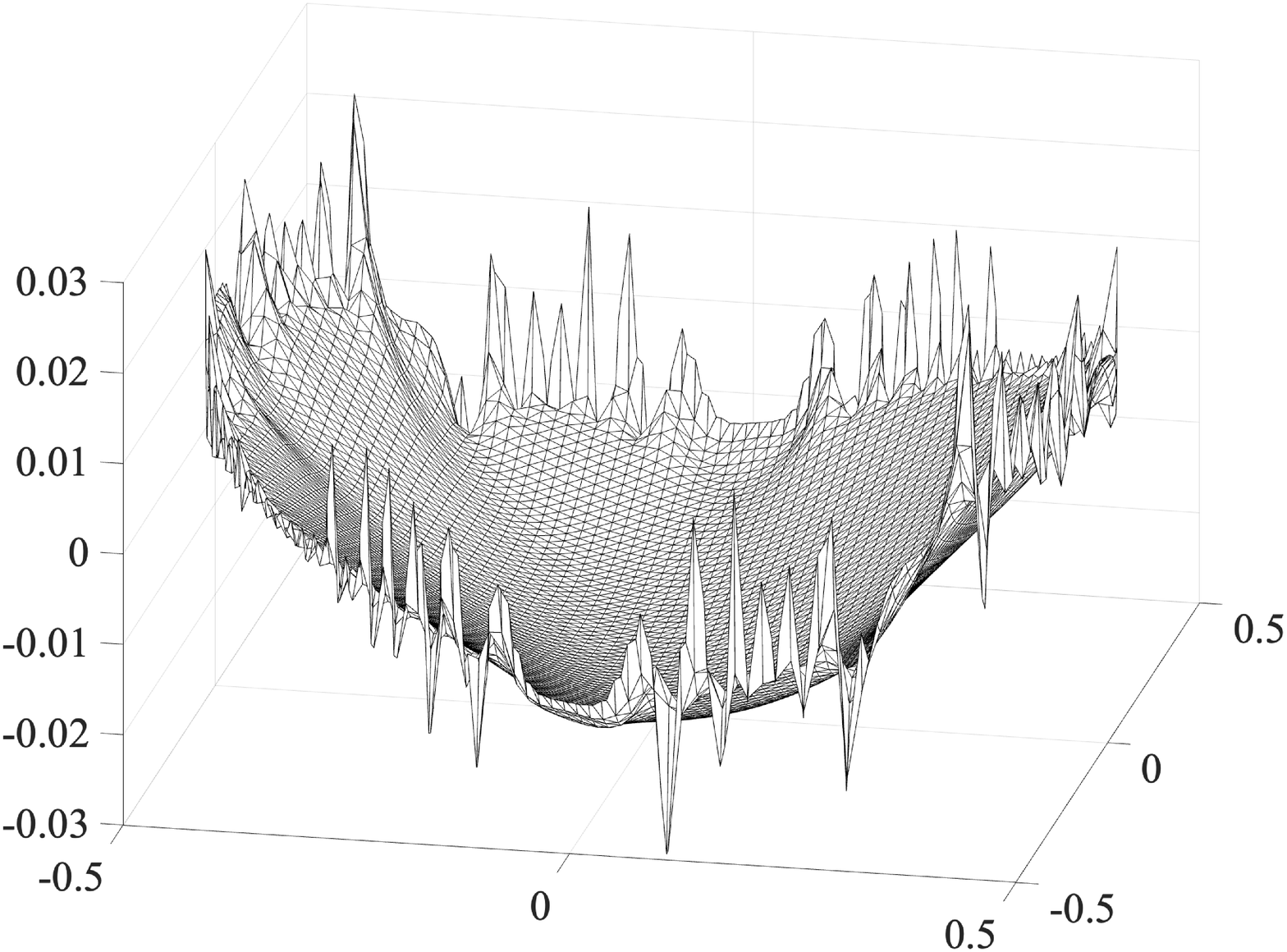}\includegraphics[scale=0.23]{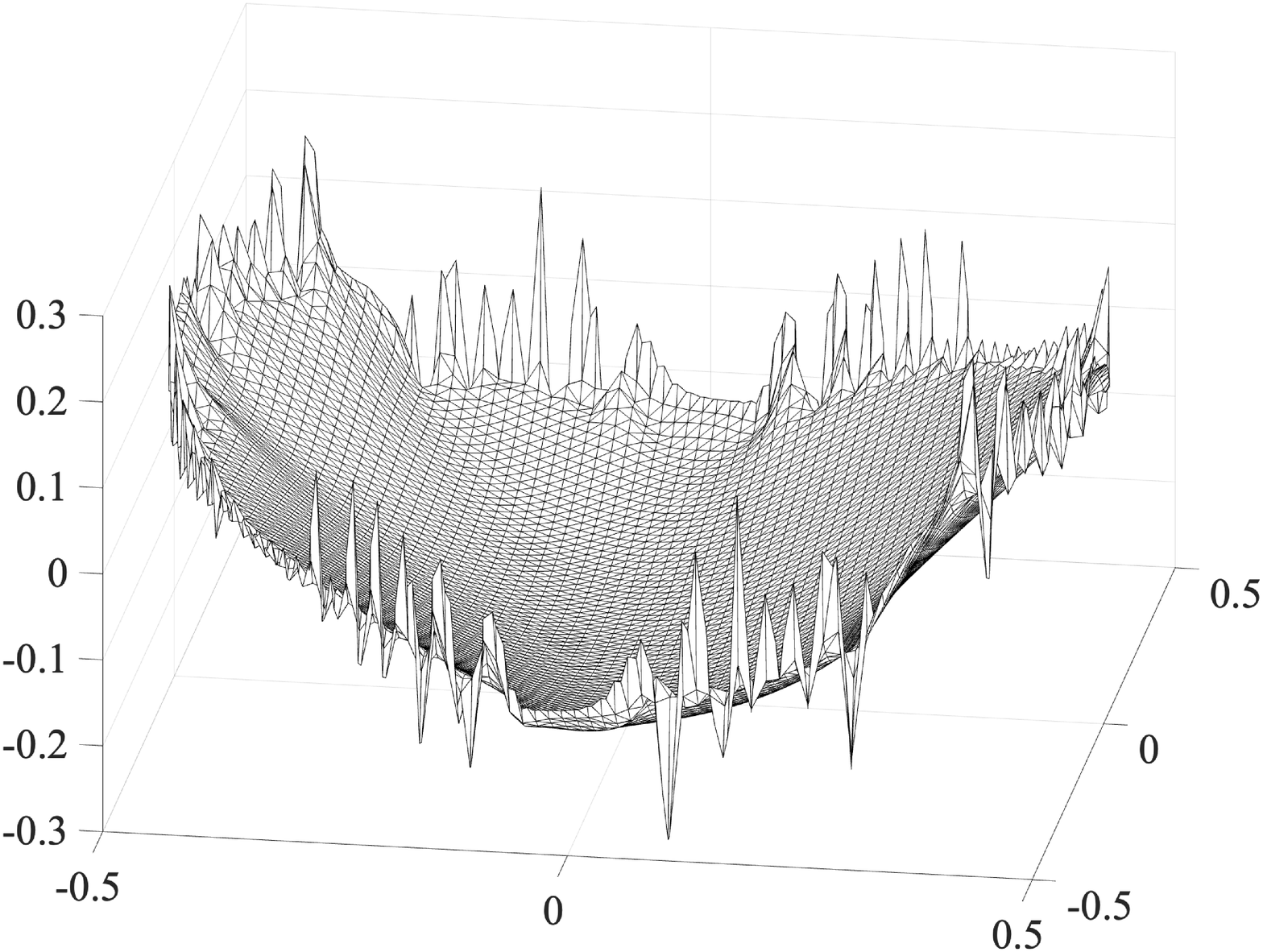}

\caption{Velocity $u_y$ on a particular mesh. $\omega = 1000$ (left) and $\omega = 10000$ (right).  \label{fig:10000}}  

\end{center}
\end{figure}

\end{document}